\documentclass{amsart}

\usepackage{amscd,amsmath,amssymb,amsfonts,bbm, calligra, xspace}
\usepackage[T1]{fontenc}
\usepackage{lmodern}
\usepackage{mathtools}
\usepackage{enumerate, enumitem}
\usepackage{multicol}
\usepackage[all]{xy}
\usepackage{mathrsfs}
\usepackage{footmisc}
\usepackage{graphicx}
\usepackage{microtype}

\usepackage{stackrel}

\usepackage{xcolor}
\definecolor{darkblue}{rgb}{0,0,0.8}
\definecolor{darkgreen}{rgb}{0,0.4,0}
\usepackage[colorlinks=true,linkcolor=darkblue, citecolor=darkgreen,  urlcolor=darkgreen, unicode,pdfborder={0 0 0},final]{hyperref}

\newtheorem{thm}{Theorem}[section]
\newtheorem{prop}[thm]{Proposition}

\newtheorem{lem}[thm]{Lemma}
\newtheorem{cor}[thm]{Corollary}

\theoremstyle{definition}

\theoremstyle{remark}
\newtheorem{rem}[thm]{Remark}
\newtheorem{rems}[thm]{Remarks}

\numberwithin{equation}{section}
\newcommand{\BM}{\mathrm{BM}}

\newcommand{\Br}{\mathrm{Br}}
\newcommand{\Gr}{\mathrm{Gr}}

\newcommand{\Id}{\mathrm{Id}}

\newcommand{\Gys}{\mathrm{Gys}}

\newcommand{\topo}{\mathrm{top}}
\newcommand{\cl}{\mathrm{cl}}
\newcommand{\Cl}{\mathrm{Cl}}

\newcommand{\qc}{\mathrm{qc}}

\newcommand{\f}{\mathrm{f}}

\newcommand{\colim}{\mathrm{colim}}

\newcommand{\et}{\mathrm{\acute{e}t}}

\newcommand{\red}{\mathrm{red}}

\newcommand{\nr}{\mathrm{nr}}

\newcommand{\tors}{\mathrm{tors}}

\newcommand{\an}{\mathrm{an}}
\newcommand{\St}{\mathrm{St}}

\newcommand{\PGL}{\mathrm{PGL}}
\newcommand{\BSL}{\mathrm{BSL}}
\newcommand{\GL}{\mathrm{GL}}

\newcommand{\Spec}{\mathrm{Spec}}

\newcommand{\Gal}{\mathrm{Gal}}

\newcommand{\BGL}{\mathrm{BGL}}

\newcommand{\h}{\mathrm{h}}

\newcommand{\RR}{\mathrm{R}}

\newcommand{\isoto}{\myxrightarrow{\,\sim\,}}
\makeatletter
\def\myrightarrow{{\setbox\z@\hbox{$\rightarrow$}\dimen0\ht\z@\multiply\dimen0 6\divide\dimen0 10\ht\z@\dimen0\box\z@}}
\def\myrightarrowfill@{\arrowfill@\relbar\relbar\myrightarrow}
\newcommand{\myxrightarrow}[2][]{\ext@arrow 0359\myrightarrowfill@{#1}{#2}}
\makeatother

\makeatletter
\newcommand{\extp}{\@ifnextchar^\@extp{\@extp^{\,}}}
\def\@extp^#1{\mathop{\bigwedge\nolimits^{\!#1}}}
\makeatother

\def\loccit{\emph{loc}.\kern3pt \emph{cit}.{}\ }
\def\eg{e.g.\kern.3em}

\def\resp {\text{resp.}\kern.3em}

\newcommand{\vsim}{\rotatebox[origin=c]{90}{$\sim$}}

\def\A{\mathbb A}

\def\Z{\mathbb Z}
\def\C{\mathbb C}

\def\K{\mathbb K}
\def\LL{\mathbb L}
\def\M{\mathbb M}
\def\Q{\mathbb Q}
\def\P{\mathbb P}
\def\R{\mathbb R}

\def\bS{\mathbb S}

\def\cA{\mathcal{A}}
\def\cN{\mathcal{N}}
\def\cM{\mathcal{M}}
\def\cL{\mathcal{L}}
\def\cO{\mathcal{O}}
\def\cE{\mathcal{E}}
\def\cF{\mathcal{F}}
\def\cG{\mathcal{G}}
\def\cH{\mathcal{H}}

\def\cM{\mathcal{M}}

\def\cI{\mathcal{I}}
\def\ci{\mathcal{C}^{\infty}}

\def\kp{\mathfrak{p}}

\def\tf{\tilde{f}}

\def\talpha{\tilde{\alpha}}
\def\tbeta{\tilde{\beta}}
\def\tgamma{\tilde{\gamma}}
\def\tdelta{\tilde{\delta}}
\def\tvarepsilon{\tilde{\varepsilon}}

\def\wS{\widetilde{S}}

\def\wH{\widetilde{H}}
\def\wK{\widetilde{K}}

\def\wOmega{\widetilde{\Omega}}

\def\whV{\widehat{V}}

\begin{document}

\title[]{\'Etale cohomology of Stein algebras}

\author{Olivier Benoist}
\address{D\'epartement de math\'ematiques et applications, \'Ecole normale sup\'erieure, CNRS,
45 rue d'Ulm, 75230 Paris Cedex 05, France}
\email{olivier.benoist@ens.fr}

\renewcommand{\abstractname}{Abstract}

\begin{abstract}
We prove that the singular cohomology with finite coefficients of a finite-dimensional Stein space~$S$ is isomorphic to the \'etale cohomology of the Stein algebra~$\cO(S)$. We deduce that any class in $H^k(S,\Z)$ comes from an algebraic variety by pullback by a holomorphic map (if $k\geq 1$), and vanishes on the complement of a nowhere dense closed analytic subset~of~$S$ (if $k\geq 2$).
\end{abstract}

\maketitle

\section{Introduction}

\subsection{A comparison theorem in Stein geometry}
\label{parcompaintro}

Let $V$ be a complex algebraic variety and let $V^{\an}=V(\C)$ be its analytification. Artin's comparison theorem (see \cite[XVI, Th\'eor\`eme 4.1]{SGA43}) states that the natural morphism
\begin{equation}
\label{compaArtin}
H^k_{\et}(V,\LL)\to H^k(V^{\an},\LL^{\an})
\end{equation}
from the \'etale cohomology of $V$ to the singular cohomology of $V^{\an}$ is an isomorphism for all $k\geq 0$ and all constructible \'etale sheaves $\LL$ on $V$ (\eg when $\LL=\Z/m$ for some~$m\geq 1$). By providing a purely algebraic description of singular cohomology groups (with finite coefficients) of complex algebraic varieties, Artin's theorem offers insights into both their algebraic and topological properties.

The main goal of this article is to prove an analogue of Artin's theorem in (complex-analytic) Stein geometry. Recall that a complex space $S$ is \textit{Stein} if the restriction map $\cO(S)\to\cO(T)$ is surjective for all discrete subsets $T$ of $S$. Examples include the affine spaces $\C^n$ for $n\geq 0$, as well as all their closed complex subspaces. 

\begin{thm}[Theorem \ref{thproperpf}]
\label{thmain}
Let $S$ be a finite-dimensional Stein space. Let $\LL$ be a constructible \'etale sheaf on a proper $\cO(S)$\nobreakdash-scheme of finite presentation $X$. Then
\begin{equation}
\label{compaintro}
H^k_{\et}(X,\LL)\to H^k(X^{\an},\LL^{\an})
\end{equation}
is an isomorphism for all $k\geq 0$.
\end{thm}

In Theorem \ref{thmain}, the complex space $X^{\an}$ is the analytification of $X$ in the sense of Bingener~\cite{Bingener}, which is naturally in bijection with $X(\C)$ where $X$ is viewed as a $\C$-scheme (see \S\ref{paranal}), and the comparison morphism \eqref{compaintro} is defined in \eqref{defcompa}. 

The hypotheses that $S$ be finite-dimensional, that $X$ be proper, that $X$ be of finite presentation (rather than merely of finite type), and that~$\LL$ be constructible are all essential for the validity of Theorem \ref{thmain} (see Remarks \ref{remscex}).

A d\'evissage based on the relative comparison theorem \cite[Theorem~4.9]{Stein} reduces Theorem \ref{thmain} to the following key particular case, which computes the \'etale cohomology of a Stein algebra (a ring of holomorphic functions on a Stein space).

\begin{thm}[Theorem \ref{thcompaS}]
\label{thmainS}
Let $S$ be a finite-dimensional Stein space. Then
$$H^k_{\et}(\Spec(\cO(S)),\Z/m)\to H^k(S,\Z/m)$$
is an isomorphism for all $k\geq 0$ and $m\geq 1$.
\end{thm}

We highlight the no less difficult particular case of Theorem \ref{thmainS} where $S=\C^n$.

\begin{cor}
One has $H^k_{\et}(\Spec(\cO(\C^n)),\Z/m)=0$ for all $k,m\geq 1$ and~${n\geq 0}$.
\end{cor}

Since a Stein space of dimension $n$ has the homotopy type of a CW complex of dimension $\leq n$ (see \cite[Korollar]{Hamm}), Theorem \ref{thmainS} implies the next corollary.

\begin{cor}
Let $S$ be a Stein space of dimension $n$. Then ${H^k_{\et}(\Spec(\cO(S)),\Z/m)}$ vanishes for all $k>n$ and $m\geq 1$.
\end{cor}

\subsection{Applications to the singular cohomology of Stein spaces}

Artin's comparison theorem \eqref{compaArtin} can be applied in two distinct ways. By computing the \'etale cohomology of complex algebraic varieties in topological terms, it can be used to give topological proofs of algebraic statements. Conversely, leveraging distinctive features of \'etale cohomology,
%Bloch-Kato, Galois actions, compatibility with affine limits of schemes
it can help study the topological properties of algebraic varieties.
Comparison theorems in Stein geometry can also be exploited in both directions. Our initial motivation to study them was to deduce, from our excellent understanding of the topology of Stein spaces, algebraic results on rings of holomorphic or meromorphic functions. Applications along these lines (of other comparison theorems discussed in \S\ref{parprevious}) appear in \cite{Stein, Steinsurface} (\eg quantitative results on the analytic version of Hilbert's 17th problem \cite[Theorem 7.5]{Stein}, or a period-index theorem in Stein geometry \cite[Theorem 8.3]{Steinsurface}). In this article, we illustrate Theorem \ref{thmain} by presenting two applications of the second kind.

First, we prove that integral cohomology classes of degree $\geq 2$ on finite-dimensional Stein spaces are supported on nowhere dense closed analytic subsets.

\begin{thm}[Theorem \ref{thintegral}]
\label{thintegralintro}
Let $S$ be a finite-dimensional Stein space. Fix a class $\alpha\in H^k(S,\Z)$ for some~$k\geq 2$. There exists a nowhere dense closed analytic subset~$Z\subset S$ such that $\alpha|_{S\setminus Z}=0$ in~$H^k(S\setminus Z,\Z)$. 
\end{thm}

One can also require $Z$ to avoid a prescribed countable subset of $S$ (see Theorem~\ref{thintegral}). Theorem \ref{thintegralintro} is easy for $k=2$ (then $\alpha$ is the first Chern class of a holomorphic line bundle~$\cL$ on $S$ and one can let $Z$ be the zero locus of a section of~$\cL$), but nontrivial for~$k\geq 3$.
We refer the reader to \S\ref{parsketchconiveau} for a sketch of its proof.

In \S\ref{parunramified}, we define \textit{unramified cohomology groups} $H^k_{\nr}(S,\Z)$ in Stein geometry, by analogy with the corresponding groups studied in algebraic geometry (see \eg \cite{BO,CTV}). We deduce from Theorem \ref{thintegralintro} that the natural morphism 
$$H^k(S,\Z)\to H^k_{\nr}(S,\Z)$$ vanishes identically for all finite-dimensional Stein spaces and all $k\geq 2$ (see Theorem~\ref{thunramified}). This shows that unramified cohomology groups behave very differently in the Stein setting compared to the algebraic one.

Second, we prove that integral cohomology classes of degree $\geq 1$ on finite-dimensional Stein spaces come from algebraic varieties. 

 \begin{thm}[Corollary \ref{corcolimZ}]
\label{thmalgebraizationintro}
Let $S$ be a finite-dimensional Stein space. Fix a class $\alpha\in H^k(S,\Z)$ for some~$k\geq 1$. There exist an affine algebraic variety $V$ over $\C$, a holomorphic map $f:S\to V^{\an}$, and a class $\beta\in H^k(V^{\an},\Z)$ such that~$\alpha=f^*\beta$.
\end{thm}

The theorem we prove is more precise (see Theorem \ref{thmcolimZ}), and has a counterpart with finite coefficients (see Theorem \ref{thmcolim}). We use the latter to endow~$H^k(S,\Z/m)$ with a canonical exhaustive increasing filtration which is functorial for holomorphic maps of Stein spaces, and which we call the \textit{Stein weight filtration} (see~\S\ref{parGS}).

\subsection{Previous works}
\label{parprevious}

Theorem \ref{thmainS} was already known for $k\leq 1$ (see \cite[Remark 6.7 (v)]{Steinsurface}) and for $k=2$ (see \cite[Theorem 4.2]{Brauerstein}), but is new for $k\geq 3$.

Prior to this article, two other comparison theorems in Stein geometry had appeared in \cite{Stein,Steinsurface}. To compare them with Theorem~\ref{thmain}, we state them in~\eqref{compaK} and~\eqref{compagen} below. Let $S$ be a finite-dimensional Stein space, let $X$ be an~$\cO(S)$\nobreakdash-scheme of finite presentation, and let $\LL$ be a constructible \'etale sheaf on~$X$.

Let $K\subset S$ be a compact subset that is \textit{Stein} in the sense that it admits a basis of Stein open neighborhoods in $S$ (for instance, $K$ could be the closed unit ball in~$S=\C^n$). Then the comparison morphisms
\begin{equation}
\label{compaK}
\underset{U}{\colim }\,H^k_{\et}(X_{\cO(U)},\LL)\to\underset{U}{\colim }\,H^k((X_{\cO(U)})^{\an},\LL^{\an}),
\end{equation}
where $U$ runs over all Stein open neighborhoods of~$K$ in $S$, are isomorphisms for all~$k\geq 0$ (see \cite[Theorem 6.1]{Stein}).

Assume that $S$ is reduced of dimension $\leq 2$. Then the comparison morphisms
\begin{equation}
\label{compagen}
\underset{a}{\colim }\,H^k_{\et}(X_{\cO(S)[\frac{1}{a}]},\LL)\to \underset{a}{\colim }\,H^k((X_{\cO(S)[\frac{1}{a}]})^{\an},\LL^{\an}),
\end{equation}
where $a\in \cO(S)$ runs over all nonzerodivisors, are isomorphisms for all $k\geq 0$ (see~\cite[Theorem 6.6]{Steinsurface}). It is not known if the restrictive hypothesis on the dimension of $S$ can be removed.

Applied to $X=\Spec(\cO(S))$, the comparison isomorphisms \eqref{compaK} and \eqref{compagen} compute in topological terms the \'etale cohomology groups of the ring $\cO(K)$ of germs of holomorphic functions on a Stein compact set $K\subset S$ (\resp of the ring $\cM(S)$ of meromorphic functions on $S$, assuming that $S$ is reduced of dimension $\leq 2$). In contrast, applying Theorem \ref{thmain} with $X=\Spec(\cO(S))$ computes the \'etale cohomology groups of the Stein algebra $\cO(S)$ itself (see Theorem \ref{thmainS}).

A crucial difference between the isomorphisms \eqref{compaK} and \eqref{compagen} on the one hand and Theorem \ref{thmain} on the other hand is that the former hold for all finitely presented~$\cO(S)$\nobreakdash-schemes $X$, whereas the latter fails in general for nonproper $X$ (see Remark \ref{remscex} (ii)). As we now explain, this complicates the proof of Theorem~\ref{thmain}.

The approach to the comparison theorems~\eqref{compaK} and~\eqref{compagen} taken in \cite{Stein,Steinsurface} originates from Artin's work~\cite[XI, \S4]{SGA43}. Its idea is to let $\varepsilon:X_{\et}\to (X^{\an})_{\cl}$ be the morphism from the small \'etale site of $X$ to the site of local isomorphisms of~$X^{\an}$ (so $\LL^{\an}=\varepsilon^*\LL$ and the comparison morphism is given by pullback by~$\varepsilon$), and to consider its Leray spectral sequence
\begin{equation}
\label{Lerayepsilon}
E_2^{p,q}=H^p_{\et}(X,\RR^q\varepsilon_*\LL^{\an})\Rightarrow H^{p+q}(X^{\an},\LL^{\an}).
\end{equation}
If $\LL\isoto \varepsilon_*\LL^{\an}$ and $\RR^q\varepsilon_*\LL^{\an}=0$ for $q>0$, then \eqref{Lerayepsilon} degenerates and yields isomorphisms $H^k_{\et}(X,\LL)\to H^k(X^{\an},\LL^{\an})$ for $k\geq 0$. To show that \eqref{compaK} (\resp \eqref{compagen}) are isomorphisms, we prove these assertions in the limit where $U$ gets smaller and smaller (\resp where~$a$ becomes more and more divisible). By design, if this strategy of proof applies to $X$, it also applies to all its Zariski-open subsets. Consequently, it cannot be used to prove a theorem that fails for nonproper $X$, such as Theorem~\ref{thmain}.

\subsection{Killing cohomology classes with finite coefficients on finite flat covers}
\label{parstrat}

We therefore had to devise a different, more global approach to Theorem \ref{thmain}. We henceforth focus on Theorem~\ref{thmainS} to which Theorem~\ref{thmain} can be reduced (as we explained in~\S\ref{parcompaintro}). 
To prove it in~\S\ref{parcompaS}, we compute both sides of
\begin{equation}
\label{compastrat}
H^k_{\et}(\Spec(\cO(S)),\Z/m)\to H^k(S,\Z/m)
\end{equation}
\`a la \v{C}ech using finite hypercoverings (this technique is referred to as \textit{cohomological descent} and goes back to \mbox{\cite[Vbis]{SGA42}}). The categories of finite hypercoverings of~$S$ and $\Spec(\cO(S))$ satisfying appropriate finiteness conditions are equivalent (this follows from \cite[Proposition~2.7]{Steinsurface}). Refining these hypercoverings ad libitum, one can deduce that \eqref{compastrat} is an isomorphism for all $k\geq 0$ from the easy case~$k=0$, if one knows how to kill classes in~$H^k_{\et}(\Spec(\cO(S)),\Z/m)$ and $H^k(S,\Z/m)$ with~$k>0$ on appropriate finite coverings. For $H^k_{\et}(\Spec(\cO(S)),\Z/m)$, this is possible by a theorem of Bhatt \cite[Theorem 1.1]{Bhatt}. That one can also achieve this for~$H^k(S,\Z/m)$ results from 
the following theorem, which is the heart of the proof of Theorem \ref{thmainS}.

\begin{thm}[Theorem \ref{thkill}]
\label{thkillintro}
Let $S$ be a finite-dimensional Stein space. For all~$k,m\geq 1$ and all~$\alpha\in H^k(S,\Z/m)$, there exist $d\geq 1$ and a holomorphic map~${p:\wS\to S}$ that is finite flat of degree $d$ such that $p^*\alpha=0$ in $H^k(\wS,\Z/m)$. 
\end{thm}

A variant of Theorem \ref{thkillintro}, restricted to a Stein compact subset $K\subset S$, and only requiring that $p$ be finite surjective (instead of finite flat of degree $d\geq 1$), was established in \cite[Proposition 3.4]{Stein} as a key step towards the comparison theorem~\eqref{compaK}. However, the argument in \cite{Stein} does not ensure the flatness~of~$p$, which is essential in the application to Theorem \ref{thmainS}. 
%Comment that in algebraic geometry, not known?
More importantly, its mechanism of proof (induction on the cohomological degree based on Mayer--Vietoris exact sequences) cannot work without a compactness hypothesis (the degree of the map~$p$ constructed in \cite{Stein} depends on the cardinality of a well-chosen open covering of~$K$; this number can be chosen finite by compactness of~$K$ but grows with~$K$).

We must therefore rely on an entirely different argument. Assume that there exists a complex manifold $\Omega$ with the homotopy type of (a sufficiently big skeleton of) the Eilenberg--MacLane space $K(\Z/m,k)$. Then $\alpha\in H^k(S,\Z/m)$ comes from the tautological class $\alpha_{\Omega}\in H^k(\Omega,\Z/m)$ by pullback by a continuous map~${f:S\to\Omega}$. Assume moreover that~$\Omega$ is \textit{Oka} (in Forstneri\v{c}'s sense \cite[Definition~1.2]{ForstnericOka} recalled in \S\ref{parprelim}), so any continuous map from a Stein space to $\Omega$ is homotopic to a holomorphic map. This allows us to choose $f$ holomorphic, which reduces the problem of killing~$\alpha$ on a finite flat cover to the same problem for~$\alpha_{\Omega}$. Given a concrete description of~$\Omega$, one can hope to solve this problem in a direct manner.

Showing the existence of Oka manifolds with a prescribed homotopy type is hard, and we do not know how to construct the required Oka approximations of Eilenberg--MacLane spaces. Instead, we apply the above strategy to an Oka manifold~$\Omega_{k,m}$ of the homotopy type of the Moore space $M(\Z/m,k)$ (see \S\ref{parMoore}). In this case, the class~$\alpha$ may not come from the tautological class~${\alpha_{k,m}\in H^k(\Omega_{k,m},\Z/m)}$ by pullback by a continuous map $f:S\to\Omega_{k,m}$. However, the obstructions to the existence of~$f$, which lie in higher degree cohomology groups of $S$, can be killed using \textit{downward} induction on the cohomological degree, allowing the proof to proceed.

To construct the $\Omega_{k,m}$, we rely on a theorem of Kusakabe \cite[Corollary~1.3]{Kusakabe} (see also the proof of Forstneri\v{c} and Wold \cite[Theorem 1.2]{FW}) stating that~$\C^N\setminus K$ is Oka if~$K$ is a polynomially convex compact subset of $\C^N$ for some~${N\geq 2}$. When~${k\geq 3}$, for appropriate choices of $N$ and~$K$, the Oka manifold ${\Omega_{k,m}:=\C^N\setminus K}$ has the homotopy type of~$M(\Z/m,k)$ (see \S\ref{parOka}). In addition, if $k\geq 5$, a careful choice of the embedding of~$K$ as a polynomially convex subset of~$\C^N$ ensures that~${\alpha_{k,m}\in H^k(\Omega_{k,m},\Z/m)}$ can be killed on a finite flat cover of~$\Omega_{k,m}$ (see \S\ref{parfiniteflatOka}). This completes the proof of Theorem~\ref{thkillintro} for~$k\geq 5$ (see \S\ref{par5}).

When $k\in\{3,4\}$, we do not know how to kill $\alpha_{k,m}$ on a finite flat cover of~$\Omega_{k,m}$. The best we can do, by means of a delicate choice of the embedding of~$K$ as a polynomially convex subset of $\C^N$, is to ensure that $\alpha_{k,m}$ lifts to an integral cohomology class, after pullback by a finite flat cover of $\Omega_{k,m}$. This allows us to reduce Theorem \ref{thkillintro} for $k\in\{3,4\}$ to the case where $\alpha$ is the reduction modulo $m$ of an integral cohomology class, which can be dealt with by further applications of Oka theory (see \S\ref{par34}). When $k\in\{1,2\}$, we cannot rely on the above strategy. Instead, we give direct proofs of Theorem \ref{thkillintro} (the case $k=1$ is trivial and, when~$k=2$, we exploit the study of Brauer groups of Stein algebras carried out in \cite{Brauerstein}; see \S\ref{par12}).

\subsection{Killing integral cohomology classes on Zariski-open subsets}
\label{parsketchconiveau}

We conclude this introduction by sketching the proof of Theorem \ref{thintegralintro}. 

A crucial input is a consequence of the Bloch--Kato conjecture (proved by Rost and Voevodsky~\cite{Voevodskyl}) according to which, if $V$ is an algebraic variety over $\C$ and~${\alpha\in H^k(V^{\an},\Z)}$ is torsion, then $V$ is covered by Zariski-open subsets~${V'\subset V}$ with~${\alpha|_{(V')^{\an}}=0}$ (Colliot-Th\'el\`ene--Voisin \cite[Th\'eor\`eme~3.1]{CTV}; in the smooth case, the argument is due to Bloch, see \cite[End of Lecture~5]{Blochbook} and \cite[Proof of Theorem~1\,(ii)]{BSrinivas}). We actually need a slight refinement of~\cite[Th\'eor\`eme~3.1]{CTV} where~$(V')^{\an}$ is required to contain a prescribed countable subset $\Xi$ of $V^{\an}$ (see~\S\ref{parBCTV}); this is essential for the application to unramified cohomology).

As a first step towards Theorem \ref{thintegralintro}, we find an affine dense open subset~$U$ of~$\Spec(\cO(S))$ such that $\alpha|_{U^{\an}}$ is~$m$\nobreakdash-torsion for some $m$ (see~\S\ref{paruptomultiple}; the argument uses Oka theory in the spirit of the strategy explained in~\S\ref{parstrat}). Then~${\alpha|_{U^{\an}}\in H^k(U^{\an},\Z)}$ is the image of a class ${\delta\in H^{k-1}(U^{\an},\Z/m)}$ by the Bockstein morphism. If one could lift $\delta$ to a class~${\tilde{\delta}\in H^{k-1}_{\et}(U,\Z/m)}$ along the comparison morphism
\begin{equation}
\label{companoniso}
H^{k-1}_{\et}(U,\Z/m)\to H^{k-1}(U^{\an},\Z/m),
\end{equation}
then one could further lift $\tilde{\delta}$ to a class in the \'etale cohomology of some affine algebraic variety $V$ over $\C$, where one could conclude by exploiting the above-mentioned \cite[Th\'eor\`eme 3.1]{CTV} (see~\S\ref{parconiveau}). Unfortunately, the morphism \eqref{companoniso} may not be surjective (Theorem \ref{thmain} does not apply because the inclusion $U\subset\Spec(\cO(S))$ is not proper). To overcome this obstacle, we ensure, after reducing to the case where~$S$ is a connected Stein manifold and through a careful choice of~$U$, that $\tilde{\delta}$ lifts to a class in the cohomology of a constructible complex of \'etale sheaves on~$\Spec(\cO(S))$ (see~\S\ref{parapplcompa}). There, our comparison theorem (Theorem \ref{thmain}) is applicable.

\subsection{Structure of the article}

We construct and study some Oka manifolds with prescribed homotopy type in Section \ref{secOka}, and we use them in Section \ref{seckill} to kill cohomology classes of Stein spaces with finite coefficients on finite flat covers, thereby proving Theorem \ref{thkillintro}. This result is applied in Section \ref{seccompa} to prove our main comparison theorem (Theorem \ref{thmain}), as well as a $\Gal(\C/\R)$-equivariant extension. Applications of the comparison theorem to the singular cohomology of Stein spaces appear in Section~\ref{secconiveau} (where we prove Theorems~\ref{thintegralintro} and study unramified cohomology) and Section \ref{secweight} (where we prove Theorem \ref{thmalgebraizationintro} and define the Stein weight filtration).

\subsection{Conventions}

We refer to \cite[1, \S 1.5]{GRCoherent} for the definition of \textit{complex spaces}. We require them to be second-countable and Hausdorff, but not necessarily reduced or finite-dimensional. By a continuous map from a complex space $S$ to a topological space, we always mean a map from the underlying topological space of $S$, thus disregarding nilpotent elements of $\cO_S$. A complex space $S$ is \textit{Stein} if~$H^k(S,\cF)=0$ for all $k\geq 1$ and all coherent sheaves $\cF$ on $S$ (see \cite{GRStein}).

An algebraic variety over $\C$ is a separated scheme of finite type over $\C$.

\section{Oka models of Moore spaces}
\label{secOka}

In this section, we construct the Oka manifolds $\Omega_{k,m}$ (see Proposition \ref{propOOka}) and study the behavior of their tautological cohomology classes $\alpha_{k,m}\in H^k(\Omega_{k,m},\Z/m)$ on finite flat covers (see Propositions \ref{liftentier} and \ref{kill}).

\subsection{Preliminaries on Oka theory}
\label{parprelim}

Theorem \ref{thOka} is (a particular case of) a result of  Forstneri\v{c} \cite{ForstnericOka}, which builds on earlier work of Gromov \cite{Gromov}. Our goal in~\S\ref{parprelim} is to dispel any doubts regarding its validity when the Stein space~$S$ is possibly not reduced, for lack of an appropriate reference. Following Forstneri\v{c} \cite[Definition~1.2]{ForstnericOka}, we say that a complex manifold $\Omega$ is \textit{Oka} if for all convex compact subsets $K\subset\C^N$ and all open neighborhoods $U$ of~$K$ in $\C^N$, any holomorphic map~$U\to \Omega$ can be approximated uniformly on $K$ by holomorphic maps~$\C^N\to \Omega$. 

\begin{thm}
\label{thOka}
Let $S$ be a Stein space. Let $p:T\to S$ be a locally trivial holomorphic fibration with Oka fibers. Then any continuous section of $p$ is homotopic to a holomorphic section of~$p$.
\end{thm}

\begin{proof}
By Proposition \ref{proplift} below, we may assume that $S$ is reduced. One can then apply the more general \cite[Theorem 1.1]{ForstnericOka}.
\end{proof}

We could not find a reference for the following lemma.

\begin{lem}
\label{lemlift}
Let $p:T\to S$ be a submersion between complex spaces. Let $S'\subset S$ be a closed subspace defined by a coherent ideal sheaf $\cI\subset\cO_S$ with $\cI^2=0$. Let~${f':S'\to T}$ be a holomorphic section of $p$ over $S'$. If $H^1(S',(f')^*T_{T/S}\otimes_{\cO_{S'}}\cI)=0$,
% (one can view $\cI$ as a coherent sheaf on~$S'$ because $\cI^2=0$)
then $f'$ can be extended to a holomorphic section $f:S\to T$ of $p$.
\end{lem}

\begin{proof}
Let $(U_i)_{i\in I}$ be a Stein open cover of $S$. Let $U_i'\subset U_i$ be the closed subspace defined by $\cI|_{U_i}$. After refining the cover $(U_i)_{i\in I}$, we may assume that $f'(U'_i)$ is included in a chart~$V_i\subset T$ with $V_i\simeq U_i\times W_i$, where $W_i$ is an open subset of~$\C^N$ with coordinates $(z_1^{(i)},\dots, z_N^{(i)})$, and where $p|_{V_i}:V_i\to U_i$ is given by the first projection~${V_i\simeq U_i\times W_i\to U_i}$. Since~$U_i$ is Stein, the map $\cO(U_i)\to\cO(U_i')$ is onto. We may thus lift the holomorphic section $f'|_{U_i'}:U_i'\to U'_i\times W_i$ of~$p$ over~$U'_i$, one coordinate of~$W_i$ at a time, to a holo\-morphic section $\tf_i:U_i\to U_i\times W_i$ of~$p$ over~$U_i$.

  On $U'_{i,j}:=U'_i\cap U'_j$, the assignment $a(f')^*db\mapsto a(\tf_i^*b-\tf_j^*b)$ (where $a$ and~$b$ are local holomorphic functions on $U_{i,j}'$ and $T$) induces a morphism of coherent sheaves $(f')^*\Omega^1_{T/S}|_{U'_{i,j}}\to \cI|_{U'_{i,j}}$. View it as an element ${\alpha_{i,j}\in H^0(U'_{i,j}, (f')^*T_{T/S}\otimes_{\cO_{S'}}\cI)}$. The family $(\alpha_{i,j})$ forms a cocycle. By vanishing of $H^1(S',(f')^*T_{T/S}\otimes_{\cO_{S'}}\cI)$, there exist ${\beta_i\in H^0(U'_i,(f')^*T_{T/S}\otimes_{\cO_{S'}}\cI)}$ with~${\alpha_{i,j}=\beta_i-\beta_j}$. View $\beta_i$ as a morphism of sheaves $(f')^*\Omega^1_{T/S}|_{U'_{i}}\to \cI|_{U'_{i}}$. Let ${f_i:U_i\to V_i\simeq U_i\times W_i}$ be the holomorphic section of $p$ over $U_i$ such that $(f_i)^*z_k^{(i)}=(\tf_i)^*z_k^{(i)}-\beta_i((f')^*dz_k^{(i)})$ for $1\leq k\leq N$. Our choices imply that the~$f_i$ glue to a holomorphic section $f:S\to T$ of $p$.
\end{proof}

\begin{prop}
\label{proplift}
Let $p:T\to S$ be a submersion of complex spaces with $S$ Stein. Any holomorphic section $f':S^{\red}\to T$ of $p$ over $S^{\red}$ can be extended to a holomorphic section $f:S\to T$.
\end{prop}
%[https://arxiv.org/pdf/0705.0591, Prop 4.1] is not far, but reducedness hypotheses there !

\begin{proof}
Let $\cN\subset\cO_S$ be the nilradical of $\cO_S$. Let~$S_k\subset S$ be the subspace defined by the coherent ideal sheaf $\cN^{k+1}\subset\cO_S$, so $S_0=S^{\red}$. Set $f_0:=f'$. Apply Lemma~\ref{lemlift} inductively to construct holomorphic sections $(f_k:S_k\to T)_{k\geq 1}$ of $p$ over $S_k$ such that~${(f_k)|_{S_{k-1}}=f_{k-1}}$.  Locally on $S$, there exists $k\geq 0$ with~$\cN^k=0$. It follows that the~$f_k$ stabilize, giving rise to the desired holomorphic section $f:S\to T$~of~$p$.
\end{proof}

\subsection{Polynomially convex simplicial complexes}

A compact subset $K\subset \C^n$ is said to be \textit{polynomially convex} if $K$ is equal to its \textit{polynomial hull}
$$\widehat{K}:=\{z\in \C^n\mid |P(z)|\leq\sup_K|P| \textrm{ for all } P\in\C[z_1,\dots,z_n]\}.$$
If $\Sigma$ is a simplicial complex, we let~$\Sigma_{\leq d}$ denote the $d$-skeleton of $\Sigma$. The next proposition is a consequence of a theorem of Vodozov and Zaidenberg \cite{VZ}. 

\begin{prop}
\label{propVZ}
Let $\Sigma$ be a finite simplicial complex of dimension $\leq n$. Then there exists an injective continuous map $f:\Sigma\to\C^{n+1}$ such that:
\begin{enumerate}[label=(\roman*)] 
\item 
\label{Si}
$f$ is piecewise linear on the first barycentric subdivision of $\Sigma$;
\item 
\label{Sii}
$f(\Sigma)\cap\C^d=f(\Sigma_{\leq d-1})$ in $\C^{n+1}$ for $0\leq d\leq n+1$;
\item
\label{Siii}
$f(\Sigma)$ is a polynomially convex subset of $\C^{n+1}$.
\end{enumerate}
\end{prop}

\begin{proof}
Let $C(\Sigma,\C)$ be the Banach algebra of continuous complex-valued functions on $\Sigma$. By \cite[Theorem 1]{VZ}, one can find $f_1,\dots,f_{n+1}\in C(\Sigma,\C)$ such that the sub-$\C$-algebra of $C(\Sigma,\C)$ generated by the $f_i$ is dense in $C(\Sigma,\C)$. The proof given in \cite[p.\,747]{VZ} is constructive and the choices made there ensure that the $f_i$ are piecewise linear on the first barycentric subdivision of $\Sigma$ and that the zero locus of~$f_{d+1}$ on $\Sigma_{\leq d}$ is~$\Sigma_{\leq d-1}$. Set $f:=(f_1,\dots, f_{n+1})$, so \ref{Si} and \ref{Sii} hold. The map~$f$ is injective because $C(\Sigma,\C)$ separates the points of $\Sigma$. The compact subset~$f(\Sigma)$ of~$\C^{n+1}$ is polynomially convex by \cite[Theorem~1.2.10]{Stout}, so \ref{Siii} also holds.
\end{proof}

\subsection{Moore spaces}
\label{parMoore}

Fix $k,m\geq 1$. Consider a CW complex with one cell of dimension $0$, one cell of dimension $k$ (yielding a $k$-sphere $\bS^k$) and one cell of dimension~$k+1$ with attaching map $\bS^k\to\bS^k$ of degree $m$. Let~$M(\Z/m,k)$ be a finite simplicial complex of dimension $k+1$ that has the homotopy type of this CW complex (see \cite[Theorem 2C.5]{Hatcher}). The space $M(\Z/m,k)$ is called a \textit{Moore space} (see \cite[Example 2.40]{Hatcher}). Using cellular (co)homology, one checks that 
\begin{equation}
\label{cohoMoore}
\wH^i(M(\Z/m,k),\Z)=0\textrm{ for }i\neq k+1\textrm{ \,\,and\,\, }\wH^{k+1}(M(\Z/m,k),\Z)=\Z/m,
\end{equation}
that $\wH_i(M(\Z/m,k),\Z)=0$ for $i\neq k$ and $\wH_k(M(\Z/m,k),\Z)=\Z/m$, and that $H^k(M(\Z/m,k),\Z/m)=\Z/m$. We also note that $\pi_i(M(\Z/m,k))=0$ for $0<i<k$.
%Nonexistence of Moore space with pi_1=H_1 in general (for k=1).

\subsection{The Oka manifolds \texorpdfstring{$\Omega_{k,m}$}{Omega k,m}}
\label{parOka}

Fix $k\geq 3$ and $m\geq 1$. Set $l:=\lfloor\frac{k}{2}\rfloor$, so $k=2l$ or~$k=2l+1$. By Proposition \ref{propVZ}, there exist injective continuous maps 
$$f_{k,m}:M(\Z/m,1)\to \C^{l+2}\textrm{ \,\,and\,\, }g_{k,m}:M(\Z/m,2)\to \C^{l+2},$$
piecewise linear on the first barycentric subdivision of $M(\Z/m,1)$ and~$M(\Z/m,2)$ respectively, such that the equalities $f_{k,m}(M(\Z/m,1))\cap\C^{l+1} =f_{k,m}(M(\Z/m,1)_{\leq l})$ and $ g_{k,m}(M(\Z/m,2))\cap\C^{l+1} =g_{k,m}(M(\Z/m,2)_{\leq l})$ hold, and whose images are polynomially convex. We define $K_{k,m}:=f_{k,m}(M(\Z/m,1))$ when~$k$ is odd and ${K_{k,m}:=g_{k,m}(M(\Z/m,2))}$ when $k$ is even. Set $\Omega_{k,m}:=\C^{l+2}\setminus K_{k,m}$. 

\begin{prop}
\label{propOOka}
For $k\geq 3$ and $m\geq 1$, the complex manifold $\Omega_{k,m}$ is Oka (in the sense recalled in \S\ref{parprelim}) and has the homotopy type of~$M(\Z/m,k)$.
\end{prop}

\begin{proof}
The complex manifold $\Omega_{k,m}$ is Oka by a theorem of Kusakabe \cite[Corollary 1.3]{Kusakabe} (see also \cite[Theorem 1.2]{FW}), because $K_{k,m}$ is polynomially convex.

The manifold $\Omega_{k,m}$ is connected (\resp simply connected). To see this, fix a path in~$\C^{l+2}$ between two points of $\Omega_{k,m}$ (\resp a based homotopy in $\C^{l+2}$ from a loop in~$\Omega_{k,m}$ to the constant loop). As $f_{k,m}$ and~$g_{k,m}$ are piecewise linear, one can perturb this path (\resp this homotopy) to ensure it avoids~$K_{k,m}$ (by transversality).

By a form of Alexander duality, for all $i\geq 1$, one has 
\begin{equation}
\label{Alexander}
H^{2l+3-i}(K_{k,m},\Z)\isoto H_{i+1}(\C^{l+2},\Omega_{k,m},\Z)\isoto H_i(\Omega_{k,m},\Z)
\end{equation}
(apply \cite[Proposition 3.4.6]{Hatcher}). We deduce from \eqref{cohoMoore} and \eqref{Alexander} that 
$$\wH_i(\Omega_{k,m},\Z)=0\textrm{ for }i\neq k\textrm{ \,\,and\,\, }\wH_k(\Omega_{k,m},\Z)=\Z/m.$$
It therefore follows from the characterization \cite[Example 4.34]{Hatcher} of Moore spaces that $\Omega_{k,m}$ has the homotopy type of $M(\Z/m,k)$.
\end{proof}

By Proposition \ref{propOOka} and \eqref{cohoMoore}, one has $H^k(\Omega_{k,m},\Z/m)=\Z/m$. We let $\alpha_{k,m}$ be a generator of $H^k(\Omega_{k,m},\Z/m)$.

\subsection{Finite flat covers of the \texorpdfstring{$\Omega_{k,m}$}{Omega k,m}}
\label{parfiniteflatOka}

Keep the notation of \S\ref{parOka}. Consider the holomorphic map $p_{k,m}:\C^{l+2}\to \C^{l+2}$, which is finite flat of degree $m$, defined by 
$$p_{k,m}(z_1,\dots, z_{l+2})=(z_1,\dots, z_{l+1},z_{l+2}^m).$$
Let $\wK_{k,m}$ and $\wOmega_{k,m}$ be the inverse images of $K_{k,m}$ and $\Omega_{k,m}$ by $p_{k,m}$. We still denote by ${p_{k,m}:\wK_{k,m}\to K_{k,m}}$ and ${p_{k,m}:\wOmega_{k,m}\to\Omega_{k,m}}$ the maps induced by $p_{k,m}$.

\begin{lem}
\label{lemCW}
For $k\geq 3$ and $m\geq 1$, the space $\wK_{k,m}$ is a CW complex and the pushforward $(p_{k,m})_*: H_{2l+2-k}(\wK_{k,m},\Z)\to H_{2l+2-k}(K_{k,m},\Z)$ is an isomorphism.
\end{lem}

\begin{proof}
If $k\geq 5$, then $K_{k,m}\subset\C^{l+1}$ and the map $p_{k,m}:\wK_{k,m}\to K_{k,m}$ is a homeomorphism. If $k\in\{3,4\}$, then $K_{k,m}\cap\C^{l+1} =(K_{k,m})_{\leq l}=(K_{k,m})_{\leq 2l+2-k}$. It follows that the CW complex structure on $K_{k,m}$ induces a CW complex structure on~$\wK_{k,m}$, with the same $(2l+2-k)$-skeleton, and with $m$ cells of dimension $2l+3-k$ in~$\wK_{k,m}$ for each cell of dimension $2l+3-k$ of $K_{k,m}$ (all attached via identical maps to the $(2l+2-k)$-skeleton). The last assertion now results from a cellular homology computation.
\end{proof}

In the proof of Proposition \ref{liftentier} below, we use Alexander duality in the form given in \cite[Corollary 11.16]{Massey}. This reference features (co)homology theories with or without compact support ($H_*^c$, $H^*_c$, $H_*^{\infty}$ and $H_{\infty}^*$, see \mbox{\cite[\S 10.1]{Massey}}). One has~$H_*^c=H_*^{\infty}$ and $H^*_c=H^*_{\infty}$ for compact spaces (see \cite[\S 10.2]{Massey}). We write~$H_*^{\BM}$ (for Borel--Moore) instead of $H_*^{\infty}$. In view of \cite[\S 8.8 and \S 9.6]{Massey}, the theories~$H^c_*$ and $H_{\infty}^*$ coincide with singular (co)homology on CW complexes. This applies in particular to $K_{k,m}$ and $\wK_{k,m}$ (see Lemma \ref{lemCW}) and to the~$\ci$ manifolds~$\Omega_{k,m}$ and $\wOmega_{k,m}$ (see \cite[Theorem 10.6]{Munkres}).

\begin{prop}
\label{liftentier}
If $k\geq 3$ and $m\geq 1$, then $(p_{k,m})^*\alpha_{k,m}\in H^k(\wOmega_{k,m},\Z/m)$ is the reduction modulo $m$ of a class $\talpha_{k,m}\in H^k(\wOmega_{k,m},\Z)$.
\end{prop}

\begin{proof}
Let $\beta\in H^{k+1}(\Omega_{k,m},\Z)[m]$ be the image of $\alpha_{k,m}$ by the boundary map of the short exact sequence $0\to\Z\xrightarrow{m}\Z\to\Z/m\to 0$. We must show that $(p_{k,m})^*\beta=0$ in $H^{k+1}(\wOmega_{k,m},\Z)$. Consider the diagram
\begin{equation}
\label{Alexanderpullback}
\begin{aligned}
\xymatrix@C=3em@R=1em{
H^{k+1}(\Omega_{k,m},\Z)\ar_{}^{(p_{k,m})^*}[d]\ar_{\cap [\Omega_{k,m}]\hspace{.5em}}^{\sim\hspace{.3em}}[r]&H^{\BM}_{2l+3-k}(\Omega_{k,m},\Z)\ar^{\sim}[r]&H_{2l+2-k}(K_{k,m},\Z)\\
H^{k+1}(\wOmega_{k,m},\Z)\ar_{\cap [\tilde{\Omega}_{k,m}]\hspace{.5em}}^{\sim\hspace{.3em}}[r]&H^{\BM}_{2l+3-k}(\wOmega_{k,m},\Z)\ar_{(p_{k,m})_*}[u]\ar^{\sim}[r]&H_{2l+2-k}(\wK_{k,m},\Z),\ar^{\raisebox{.3em}{\vsim}}_{(p_{k,m})_*}[u]
}
\end{aligned}
\end{equation}
whose left horizontal arrows are Poincar\'e duality isomorphisms given by cap products with the fundamental classes of $\Omega_{k,m}$ and $\wOmega_{k,m}$ (see \cite[Theorem~11.4]{Massey}), whose right horizontal arrows are boundary maps of long exact sequences of pairs (see~\cite[(4a)\,p.\,86]{Massey}) which are isomorphisms by vanishing of $H^{\BM}_{2l+2-k}(\C^{l+2},\Z)$ and $H^{\BM}_{2l+3-k}(\C^{l+2},\Z)$, whose right square commutes by \cite[(4c)\,p.\,86]{Massey}, and whose right vertical arrow is an isomorphism by Lemma~\ref{lemCW}. One computes that
$$(p_{k,m})_*(p_{k,m}^*(\beta)\cap[\wOmega_{k,m}])=\beta\cap(p_{k,m})_*[\wOmega_{k,m}]=m\cdot (\beta\cap [\Omega_{k,m}])=0,$$
where we used the naturality of the cap product (see \cite[(A1)\,p.\,325]{Massey}), the fact that $p_{k,m}$ has degree~$m$, and that $\beta$ is $m$-torsion. It now follows from diagram~\eqref{Alexanderpullback} (and in particular the commutativity of its right square) that~${(p_{k,m})^*\beta=0}$.
\end{proof}

\begin{prop}
\label{kill}
If $k\geq 5$ and $m\geq 1$, then $(p_{k,m})^*\alpha_{k,m}=0\textrm{ in }H^k(\wOmega_{k,m},\Z/m)$.
\end{prop}

\begin{proof}
As $k\geq 5$, one has $K_{k,m}\subset \C^{l+1}$ in $\C^{l+2}$. It follows that $\wK_{k,m}=K_{k,m}$, hence that $\wOmega_{k,m}=\Omega_{k,m}$ (as subsets of $\C^{l+2}$). Using Proposition \ref{propOOka} and~\eqref{cohoMoore} yields
$$H^k(\wOmega_{k,m},\Z)=H^k(\Omega_{k,m},\Z)=H^k(M(\Z/m,k),\Z)=0.$$
The class $\talpha_{k,m}\in H^k(\wOmega_{k,m},\Z)$ given by Proposition \ref{liftentier} therefore vanishes, and it follows that $(p_{k,m})^*\alpha_{k,m}=0$.
\end{proof}

\section{Killing singular cohomology on finite flat covers}
\label{seckill}

The main goal of this section is the following theorem.

\begin{thm}
\label{thkill}
Let $S$ be a finite-dimensional Stein space. For all $k,m\geq 1$ and all~$\alpha\in H^k(S,\Z/m)$, there exist $d\geq 1$ and a holomorphic map $p:\wS\to S$ that is finite flat of degree $d$ such that $p^*\alpha=0$ in $H^k(\wS,\Z/m)$. 
\end{thm}

\begin{rem}
It follows from the proof of Theorem \ref{thkill} that the integer $d$ only depends on $m$ and on the dimension of $S$. It also follows from this proof that the finite map~$p$ can be chosen to be locally of complete intersection, in addition to being flat.
\end{rem}

After gathering general facts about Severi--Brauer spaces in \S\ref{parSB}, we prove Theorem \ref{thkill} in \S\ref{par12} when $k\in\{1,2\}$, in \S\ref{par5} when $k\geq 5$ (this is the heart of the proof), and in \S\ref{par34} for the slightly more involved cases where~${k\in\{3,4\}}$.

\subsection{Severi--Brauer spaces}
\label{parSB}

Let $S$ be a complex space. A \textit{Severi--Brauer space} of relative dimension $N$ over $S$ is a locally trivial holomorphic fibration $\pi:X\to S$ whose fibers are isomorphic to $\P^N(\C)$.  As the group of holomorphic automorphisms of $\P^{N}(\C)$ is $\PGL_{N+1}(\C)$, Severi--Brauer spaces of relative dimension $N$ over $S$ are classified by $H^1(S,\PGL_{N+1}(\cO_S))$. The following two lemmas are classical in algebraic geometry (in particular, Lemma \ref{trivSB} is a complex-analytic version of Ch\^atelet's theorem \cite[p.\,35]{Chatelet}), and we include their proofs for completeness. 

\begin{lem}
\label{trivSB}
Let $\pi:P\to S$ be a Severi--Brauer space over a complex space $S$. If~$\pi$ has a holomorphic section, then $\pi$ is the projectivization of a holomorphic vector bundle on~$S$.
\end{lem}

\begin{proof}
Since $\cO_S\isoto\pi_*\cO_P$, one has $\cO_S^{\times}\isoto\pi_*\cO_P^{\times}$. The Leray spectral sequence
$$E^{p,q}_2=H^p(S,\RR^q\pi_*\cO_P^{\times})\Rightarrow H^{p+q}(P,\cO_P^{\times})$$ 
therefore gives rise to an exact sequence
\begin{equation}
\label{5terms}
H^1(P,\cO_P^{\times})\to H^0(S,\RR^1\pi_*\cO_P^{\times})\to H^2(S,\cO_S^{\times})\to H^2(P,\cO_P^{\times}).
\end{equation}
That $\pi$ has a holomorphic section shows that $H^2(S,\cO_S^{\times})\to H^2(P,\cO_P^{\times})$ is injective. 
If $U$ is a contractible Stein open subset of $S$ that trivializes~$\pi:P\to S$, then 
$$H^1(\pi^{-1}(U),\cO_{\pi^{-1}(U)}^{\times})=H^1(\P^N(\C)\times U,\cO_{\P^N(\C)\times U}^{\times})=\Z$$
(use the exponential short exact sequence \cite[Lemma p.\,142]{GRStein} on $\P^N(\C)\times U$). We deduce an isomorphism $\RR^1\pi_*\cO_P^{\times}\to\Z$ given by the degree of line bundles on the fibers of $\pi$. It now follows from~\eqref{5terms} that there is a holomorphic line bundle~$\cL$ on $P$ restricting to~$\cO_{\P^N(\C)}(1)$ on the fibers of $\pi$. One then checks locally (over open subsets~$U$ of $S$ as above) that~$\pi_*\cL$ is locally free and that the natural holomorphic map~${P\to\P(\pi_*\cL)}$ over $S$ is an isomorphism (use cohomology and base change in its form \cite[III, Corollary~3.9]{BS}).
\end{proof}

\begin{lem}
\label{lemBertini}
Let $S$ be a finite-dimensional Stein space. Let $\pi:P\to S$ be a Severi--Brauer space of relative dimension $N$. There exist $d\geq 1$ and a closed complex subspace $T\subset P$ such that~$\pi|_T:T\to S$ is finite flat of degree $d$.
\end{lem}

\begin{proof}
Fix $e\geq 1$. Let $\cL:=\cO_P(-K_{P/S})$ be the relative anticanonical line bundle of~$\pi$, which restricts to $\cO_{\P^N(\C)}(N+1)$ on the fibers of $\pi$. By cohomology and base change (see \eg \cite[III, Corollary~3.9]{BS}), the sheaf $\cE:=\pi_*(({\cL^{\otimes e})^{\oplus N}})$ is locally free and, letting $f:E\to S$ be the geometric vector bundle associated with~$\cE$, the fiber~$E_s$ of~$f$ over any~$s\in S$ can be identified with~$H^0(\P^{N}(\C),\cO_{\P^N(\C)}(e(N+1)))^{\oplus N}$ (this identification is well defined up to the natural  action of $\GL_{N+1}(\C)$).

Consider the Zariski-open subset~$\Theta$ of~$H^0(\P^{N}(\C),\cO_{\P^N(\C)}(e(N+1)))^{\oplus N}$ consisting of $N$-tuples~$(\sigma_1,\dots,\sigma_N)$ such that~${\{\sigma_1=\dots=\sigma_N=0\}}$ is a complete intersection in $\P^N(\C)$. As $\Theta$ is~$\GL_{N+1}(\C)$\nobreakdash-invariant, there exists an open subset $E^0\subset E$ such that $E^0\cap E_s$ is sent to $\Theta$ by the above identifications, for all $s\in S$. The restriction~$f^0:E^0\to S$ of $f$ is a locally trivial holomorphic fibration with fiber $\Theta$.

Let $n$ be the dimension of $S$. By \cite[Corollary 3.2]{BD}, if $e$ has been chosen big enough, then the complement of $\Theta$ in $H^0(\P^{N}(\C),\cO_{\P^N(\C)}(e(N+1)))^{\oplus N}$ has complex codimension $>\frac{n}{2}$, and it follows that~$\pi_i(\Theta)=0$ for $1\leq i\leq n-1$. As~$S$ has the homotopy type of a CW complex of dimension $\leq n$ (see \cite[Korollar]{Hamm}), one can therefore use obstruction theory to find a continuous section $g:S\to E^0$ of~$f^0$ (apply \cite[Theorem 34.2]{Steenrod}). 
As $\Theta$ is an Oka manifold (see \cite[Proposition~5.6.10]{Forstneric}), Theorem~\ref{thOka} shows that~$g$ can be chosen to be holomorphic.

The section $g$ corresponds to a section $(\sigma_1,\dots,\sigma_N)\in H^0(S,\cE)=H^0(P,\cL^{\otimes e})^{\oplus N}$. Define $T:=\{\sigma_1=\dots=\sigma_N=0\}\subset P$. By construction, the complex subspace~$T$ of~$P$ is a relative complete intersection of codimension $N$ over $S$ and hence~${\pi|_T:T\to S}$ is finite flat of degree $d:=(e(N+1))^N$.
\end{proof}

\subsection{Classes of degree \texorpdfstring{$1$}{1} or \texorpdfstring{$2$}{2}}
\label{par12}

Lemmas \ref{lemkill2a} and \ref{lemkill2b} below will be combined to prove Proposition \ref{kill12}.

\begin{lem}
\label{lemkill2a}
Let $S$ be a finite-dimensional Stein space. Fix $\beta\in H^3(S,\Z)[m]$ for some $m\geq 1$. There exist $d\geq 1$ and a holomorphic map $p:T\to S$ that is finite flat of degree $d$ with $p^*\beta=0$ in $H^3(T,\Z)[m]$. 
\end{lem}

\begin{proof}
Let~$\Br_{\an}(S)$ and $\Br_{\topo}(S)$ be the analytic and topological Brauer groups of~$S$, based on holomorphic and topological Azumaya algebras on $S$ (see~\mbox{\cite[\S 2.1]{Brauerstein}}). There are natural isomorphisms
\begin{equation}
\label{Brauer}
\Br_{\an}(S)\isoto\Br_{\topo}(S)\isoto H^3(S,\Z)_{\tors}
\end{equation}
(see \cite[Theorem 3.5]{Brauerstein}). Consider a holomorphic Azumaya algebra $\cA$ on $S$ whose class in $H^3(S,\Z)_{\tors}$ (see \eqref{Brauer}) is equal to $\beta$. Let $N+1$ be the degree of~$\cA$, so~$\cA$ is classified by an element $\alpha\in H^1(S,\PGL_{N+1}(\cO_S))$ (see \cite[\S 2.1]{Brauerstein}). Let~$\pi:P\to S$ be the Severi--Brauer space associated with $\alpha$ (see \S\ref{parSB}). By Lemma~\ref{lemBertini}, there exists a complex subspace $T\subset P$ such that $p:=\pi|_T:T\to S$ is finite flat of some degree~${d\geq 1}$. It follows from Lemma \ref{trivSB} that $p^*\alpha\in H^1(T,\PGL_{N+1}(\cO_T))$ lifts to~$H^1(T,\GL_{N+1}(\cO_T))$. We deduce that the image of $p^*\alpha$ in $H^2(T,\cO_T^{\times})$ vanishes, so the class of $p^*\cA$ in $\Br_{\an}(T)$ vanishes (see \cite[(2.1)]{Brauerstein}), and hence that $p^*\beta=0$.
\end{proof}

\begin{lem}
\label{lemkill2b}
Let $S$ be a finite-dimensional Stein space. Fix $\gamma\in H^2(S,\Z)$. For all~${m\geq 1}$, there exist $d\geq 1$ and a holomorphic map $p:\wS\to S$ that is finite flat of degree $d$ such that $p^*\gamma$ is divisible by $m$ in $H^2(\wS,\Z)$.
\end{lem}

\begin{proof}
The cohomology long exact sequence associated with the exponential short exact sequence $0\to\Z\xrightarrow{2\pi i}\cO_S\xrightarrow{\exp}\cO_S^{\times}\to 0$ (for which see \cite[Lemma p.\,142]{GRStein}) and the vanishing of $H^2(S,\cO_S)$ (because~$S$ is Stein) show that $\gamma$ is the first Chern class of a holomorphic line bundle~$\cL$ on $S$. 

Let $n$ be the dimension of $S$. By \cite[Theorem~1]{Kripke}, the line bundle $\cL$ (which is globally generated since $S$ is Stein) is generated by $n+1$ global sections~$\sigma_0,\dots,\sigma_n$. Let~${\sigma:S\to \P^n(\C)}$ be the holomorphic map defined by ${\sigma(s):=[\sigma_0(s):\dots:\sigma_n(s)]}$, so $\sigma^*\cO_{\P^n(\C)}(1)\simeq\cL$. The degree $m^n$ finite flat holomorphic map $r:\P^n(\C)\to\P^n(\C)$ given by ${r([x_0:\dots:x_n])=[x_0^m:\dots:x_n^m]}$ is such that $r^*\cO_{\P^n(\C)}(1)\simeq\cO_{\P^n(\C)}(m)$. Define $p:\wS\to S$ to be the base change of $r$ by $\sigma$. Then $p^*\cL$ is the $m$-th tensor power of some holomorphic line bundle on $\wS$, and the lemma is proved.
\end{proof}

\begin{prop}
\label{kill12}
Theorem \ref{thkill} holds for $k\in\{1,2\}$.
\end{prop}

\begin{proof}
When $k=1$, one can take $p:\wS\to S$ to be the unramified cyclic cover of degree $m$ that the class $\alpha\in H^1(S,\Z/m)$ classifies. If $k=2$, we let $\beta\in H^3(S,\Z)[m]$ be the image of $\alpha\in H^2(S,\Z/m)$ by the boundary map of the long exact sequence of cohomology associated with $0\to\Z\xrightarrow{m}\Z\to\Z/m\to 0$. Lemma \ref{lemkill2a} shows that, after replacing $S$ with a finite flat cover, we may assume that $\beta=0$. Then $\alpha$ is the reduction modulo $m$ of a class $\gamma\in H^2(S,\Z)$ to which one can apply Lemma~\ref{lemkill2b}.
\end{proof}

\subsection{Classes of degree \texorpdfstring{$\geq 5$}{at least 5}}
\label{par5}

Recall that $\alpha_{k,m}$ was defined in \S\ref{parOka} to be a generator of ${H^k(\Omega_{k,m},\Z/m)\simeq\Z/m}$. Let $\varphi_{k,m}:\Omega_{k,m}\to K(\Z/m,k)$ be a continuous map classifying $\alpha_{k,m}$. Let $F_{k,m}$ be the homotopy fiber of~$\varphi_{k,m}$. 

\begin{lem}
\label{homotofi}
Fix $k\geq 3$ and $m\geq 1$. Then $\pi_i(F_{k,m})=0$ for $1\leq i\leq k$ and $\pi_i(F_{k,m})$ is finite for $i>k$. 
\end{lem}

\begin{proof}
As $\varphi_{k,m}$ induces an isomorphism on $H^k(-,\Z/m)$, it induces an isomorphism on $H_k(-,\Z)$ by the universal coefficient theorem, hence on $\pi_k(-)$ by the Hurewicz theorem (noting that $\pi_i(\Omega_{k,m})=0$ for $0<i<k$ by Proposition \ref{propOOka} and \S\ref{parMoore}). As the homology groups of $\Omega_{k,m}$ in positive degree are finite (see Proposition \ref{propOOka} and~\S\ref{parMoore}), so are its homotopy groups by \cite[Corollaire 2 p.\,274]{Serrehomotopie}. The long exact sequence of homotopy groups of a fibration completes the proof of the lemma.
\end{proof}

\begin{prop}
\label{kill5}
Theorem \ref{thkill} holds for $k\geq 5$.
\end{prop}

\begin{proof}
Let $n$ be the dimension of $S$. We henceforth fix $n$ and argue by downward induction on $k$. If $k>n$, then $H^k(S,\Z/m)=0$ because $S$ has the homotopy type of a CW complex of dimension $\leq n$ (see \cite[Korollar]{Hamm}), so the result is proven in this case. 
We assume from now on that $5\leq k\leq n$.

 Let $\varphi:S\to K(\Z/m,k)$ be a continuous map classifying~$\alpha\in H^k(S,\Z/m)$. Consider the problem of finding a continuous map $\psi:S\to \Omega_{k,m}$ such that~$\varphi_{k,m}\circ\psi$ is homotopic to $\varphi$. Obstruction theory provides successive obstructions for this problem, which live in $H^{i+1}(S,\pi_i(F_{k,m}))$ (see \cite[Theorem 34.2]{Steenrod}). These obstructions vanish for~${i\leq k}$ by Lemma \ref{homotofi} and for $i\geq n$ because $S$ has the homotopy type of a CW complex of dimension~$\leq n$ (see \cite[Korollar]{Hamm}).  As~$\pi_i(F_{k,m})$ is finite for~${k<i<n}$ (see Lemma~\ref{homotofi}), one can kill them after replacing $S$ with a finite flat cover of positive degree, by repeated applications of the induction hypothesis.

We may therefore assume that there exists a continuous map $\psi:S\to \Omega_{k,m}$ such that $\psi^*\alpha_{k,m}=\alpha$. As $\Omega_{k,m}$ is Oka by Proposition \ref{propOOka}, it follows from Theorem~\ref{thOka} that there exists a holomorphic map $\chi:S\to \Omega_{k,m}$ that is homotopic to $\psi$, and hence such that $\chi^*\alpha_{k,m}=\alpha$. By Proposition~\ref{kill}, letting $p:\wS\to S$ be the pullback by $\chi$ of the finite flat holomorphic map~$p_{k,m}:\wOmega_{k,m}\to\Omega_{k,m}$ concludes the proof.
\end{proof}

\subsection{Classes of degree \texorpdfstring{$3$}{3} or \texorpdfstring{$4$}{4}}
\label{par34}

\begin{prop}
\label{kill4}
Theorem \ref{thkill} holds for $k=4$. 
\end{prop}

\begin{proof}
Fix $\alpha\in H^4(S,\Z/m)$. Arguing as in the proof of Proposition \ref{kill5} (and using Proposition \ref{kill5} to kill the relevant obstructions on finite flat covers of~$S$), we reduce to the case where there exists a holomorphic map $\chi:S\to\Omega_{4,m}$ such that~$\chi^*\alpha_{4,m}=\alpha$. By Proposition \ref{liftentier}, replacing $S$ with the pullback by $\chi$ of the finite flat holomorphic map~$p_{4,m}:\wOmega_{4,m}\to\Omega_{4,m}$ allows us to assume that $\alpha$ is the reduction modulo $m$ of a class~$\gamma\in H^4(S,\Z)$.

Let $\eta:\BSL_2(\C)\to K(\Z,4)$ be a continuous map classifying the second Chern class~$c_2$. As $H^*(\BSL_2(\C),\Z)=\Z[c_2]$, and given the computation of $H^*(K(\Z,4),\Q)$ in \cite[Proposition 4 p.\,501]{Serrehomologie}, the map $\eta$ induces an isomorphism on $H^*(-,\Q)$, hence on $H_*(-,\Q)$, hence also on~$\pi_*(-)\otimes{\Q}$ by \cite[Th\'eor\`eme 3]{Serrehomotopie}. In addition, by the Hurewicz theorem, the map $\eta$ induces an isomorphism on $\pi_i(-)$ for $i\leq 4$. Letting~$F_{\eta}$ be the homotopy fiber of $\eta$, the long exact sequence of homotopy groups of a fibration shows that~$\pi_i(F_{\eta})=0$ for~$1\leq i\leq 4$ and that $\pi_i(F_{\eta})$ is finite for~$i>4$. 

Let $\varphi:S\to K(\Z,4)$ be a continuous map classifying $\gamma$. The successive obstructions to lifting~$\varphi$ to $\BSL_2(\C)$ (see \cite[Theorem 34.2]{Steenrod}) live in~$H^{i+1}(S,\pi_i(F_{\eta}))$ for~$i\geq 5$, where~$\pi_i(F_{\eta})$ is finite. By Proposition \ref{kill5}, one can kill these obstructions on finite flat covers of~$S$. This allows us to reduce to the case where~$\gamma=c_2(\cE)$ for some rank~$2$ topological complex vector bundle $\cE$ on~$S$. Use Grauert's Oka principle \cite[Satz~II]{GrauertOka} to endow~$\cE$ with the structure of a holomorphic vector bundle (the reference \cite{GrauertOka} assumes $S$ to be reduced, but one can reduce to this case by applying Proposition \ref{proplift} to a trivial fibration whose fiber is a Grassmannian).

Let $\pi:\P(\cE)\to S$ be the projectivization of $\cE$. Use Lemma \ref{lemBertini} to find a complex subspace $T\subset \P(\cE)$ such that $\pi|_{T}:T\to S$ is finite flat of some positive degree. The vector bundle $\pi^*\cE$ is tautologically an extension of two holomorphic line bundles~$\cL$ and $\cL'$ on $\P(\cE)$. It follows that $(\pi|_T)^*c_2(\cE)=c_1(\cL|_T)c_1(\cL'|_T)$ in~$H^4(T,\Z)$. Replacing $S$ with $T$ and $\cE$ with $(\pi|_T)^*\cE$, we may assume that $\gamma$ is the product of two classes in $H^2(T,\Z)$, hence that $\alpha$ is the product of two classes in $H^2(T,\Z/m)$. The proposition now follows from the $k=2$ case dealt with in Proposition \ref{kill12}.
\end{proof}

\begin{prop}
\label{kill3}
Theorem \ref{thkill} holds for $k=3$. 
\end{prop}

\begin{proof}
Fix $\alpha\in H^3(S,\Z/m)$. Arguing as in the first paragraph of the proof of Proposition \ref{kill4}, we may assume that $\alpha$ is the reduction modulo $m$ of a class~${\gamma\in H^3(S,\Z)}$.

Set $\Omega:=\C^2\setminus\{(0,0)\}$, so $\Omega$ has the homotopy type of $\bS^3$. Let $\theta:\Omega\to K(\Z,3)$ be a continuous map classifying a generator $\delta\in H^3(\Omega,\Z)\simeq \Z$. In view of \cite[Proposition 4 p.\,501]{Serrehomologie}, the map $\theta$ induces an isomorphism on~$H^*(-,\Q)$, hence on~$H_*(-,\Q)$, hence on~$\pi_*(-)\otimes{\Q}$ by \cite[Th\'eor\`eme 3]{Serrehomotopie}. In addition, by the Hurewicz theorem, the map $\theta$ induces an isomorphism on $\pi_i(-)$ for $i\leq 3$. Letting~$F_{\theta}$ be the homotopy fiber of $\theta$, the long exact sequence of homotopy groups of a fibration shows that~$\pi_i(F_{\theta})=0$ for~$1\leq i\leq 3$ and that $\pi_i(F_{\theta})$ is finite for~$i>3$. 

Let $\varphi:S\to K(\Z,3)$ be a continuous map classifying $\gamma$. The successive obstructions to lifting~$\varphi$ to $\Omega$ (see \cite[Theorem 34.2]{Steenrod}) live in~$H^{i+1}(S,\pi_i(F_{\theta}))$ for~$i\geq 4$, where $\pi_i(F_{\theta})$ is finite. Killing these obstructions on finite flat covers of $S$ using Proposition \ref{kill5}, we may assume that there exists a continuous map~$\psi:S\to \Omega$ with~$\psi^*\delta=\gamma$. As $\Omega$ is Oka (see \eg \cite[Proposition 5.6.1]{Forstneric}), one can apply Theorem \ref{thOka} to find a holomorphic map $\chi:S\to \Omega$ that is homotopic to $\psi$, so~$\chi^*\delta=\gamma$.

Let $q_m:\Omega\to\Omega$ be the  holomorphic map~$(z_1,z_2)\mapsto(z_1,z_2^m)$, which is finite flat of degree $m$. One has $(q_m)^*\delta=m\delta$. Let $p:\wS\to S$ be the pullback of~$q_m$ by $\chi$. Then $p^*\gamma\in H^3(\wS,\Z)$ is a multiple of $m$ and hence $p^*\alpha=0$.
\end{proof}

\section{The comparison theorem}
\label{seccompa}

After providing the necessary background on analytification and Grothendieck topologies in Stein geometry in \S\ref{paranal} and \S\ref{parGroth}, we use Theorem \ref{thkill} to prove our key comparison theorem (Theorem \ref{thcompaS}) in \S\ref{parcompaS}. We deduce its constructible and $\Gal(\C/\R)$\nobreakdash-equivariant extensions in \S\ref{parconstructible} and \S\ref{parGeq} respectively. The construction of the analytic weight filtration is derived in \S\ref{parGS}.

\subsection{Analytification in Stein geometry}
\label{paranal}

Let $S$ be a Stein space and let $X$ be a separated
%for X^an to be separated.
$\cO(S)$\nobreakdash-scheme locally of finite presentation. In \cite[Satz 1.1]{Bingener},  Bingener functorially associates with $X$ its \textit{analytification}: a complex space $X^{\an}$ over $S$ endowed with a morphism ${i_X:X^{\an}\to X}$. Concretely, if ${X=\{f_1=\dots=f_M=0\}\subset\A^N_{\cO(S)}}$ is affine, then $X^{\an}=\{f_1=\dots=f_M=0\}\subset S\times \C^N$. In general, the construction proceeds by gluing analytifications of affine charts. The space $X^{\an}$ is~naturally in bijection with $X(\C)$, where $X$ is viewed as a $\C$-scheme (see \cite[Lemma~2.3]{Steinsurface}).
Analytification is compatible with fiber products (see \cite[p.\,3]{Bingener}) and restriction of scalars (see \cite[Lemma 2.2]{Steinsurface}). The \textit{analytification} of a finitely presented quasi-coherent sheaf $\cF$ on $X$ is the coherent sheaf~$\cF^{\an}:=i_X^*\cF$ on~$X^{\an}$ (see~\mbox{\cite[p.\,3]{Bingener}}). 

Let $S$ be a Stein space. A coherent sheaf $\cG$ on $S$ is said to be \textit{globally finitely presented} if it admits a presentation of the form $\cO_S^{\oplus N'}\to\cO_S^{\oplus N}\to\cG\to 0$ for some~$N,N'\geq 0$ (we warn the reader that the reference \cite[\S 2.3]{Steinsurface} calls such sheaves \textit{finitely presented}). A finite holomorphic map $p:T\to S$ is said to be \textit{globally finitely presented} if the coherent sheaf $p_*\cO_T$ is globally finitely presented.

The next proposition, in the spirit of \cite{Forster}, follows from \cite[Proposition 2.5 and Remarks~2.6]{Steinsurface} (see~\mbox{\cite[(2.2)]{Brauerstein}}). It will be used in the proof of Proposition~\ref{propconiveauuptomultiple}.

\begin{prop}
\label{propmonoidaleq}
Let $S$ be a Stein space. There is an exact monoidal adjoint equivalence of categories
\begin{equation}
\label{monoidaleq}
 \left\{  \begin{array}{l}
    \textrm{finitely presented quasi-coherent}\\ \hspace{2em}\textrm{sheaves on } \Spec(\cO(S))
  \end{array}\right\}
    \stackrel[]{}{\rightleftarrows} 
 \left\{  \begin{array}{l}
    \textrm{globally finitely presented}\\ \hspace{.7em}\textrm{coherent sheaves on }S
  \end{array}\right\}
\end{equation}
with left-to-right functor $\cF\mapsto\cF^{\an}$ and right-to-left functor $\cG\mapsto \widetilde{\cG(S)}$.
If $S$ is finite-dimensional, then \eqref{monoidaleq} restricts to a rank-preserving equivalence between the categories of vector bundles on $\Spec(\cO(S))$ and holomorphic vector bundles on $S$.
\end{prop}

The following proposition is used in the proof of Theorem \ref{thcompaS}.

\begin{prop}
\label{propadjeq}
If $S$ is a Stein space, then the two functors $X\mapsto X^{\an}$ and ${T\mapsto\Spec(\cO(T))}$ induce an adjoint equivalence
\begin{equation}
\label{adjeq}
 \left\{  \begin{array}{l}
    \textrm{finite }\cO(S)\textrm{-schemes }X\\ \hspace{.3em}\textrm{of finite presentation}
  \end{array}\right\}
    \rightleftarrows
 \left\{  \begin{array}{l}
    \hspace{2em}\textrm{globally finitely presented}\\ \textrm{finite holomorphic maps } p:T\to S
  \end{array}\right\}.
\end{equation}

If the $\cO(S)$\nobreakdash-scheme~$X$ is flat (\resp surjective) then so is $p$.

If $p$ is flat (\resp flat and surjective) then so is the $\cO(S)$-scheme $X$.
\end{prop}

\begin{proof}
This is proved in \cite[Proposition 2.7]{Steinsurface}, except for the parenthesized assertions. If $f:X\to\Spec(\cO(S))$ is a surjective morphism of finite presentation, then~$f^{\an}:X^{\an}\to S$ is surjective by \cite[Satz 3.1 (2)]{Bingener}. Now let $p:T\to S$ be a surjective globally finitely presented finite flat holomorphic map with induced finite flat morphism of finite presentation $f:\Spec(\cO(T))\to\Spec(\cO(S))$. The image~$Z$ of~$f$ is a closed and open subscheme of $\Spec(\cO(S))$. As the surjective map $p$ factorizes through the open subspace $Z^{\an}$ of $S$, one has $Z^{\an}=S$. It therefore follows from the equivalence~\eqref{adjeq} that $Z=\Spec(\cO(S))$ and hence that $f$ is surjective.
\end{proof}

\begin{rem}
\label{remnonsurj}
In the setting of Proposition \ref{propadjeq}, if $p$ is not assumed to be flat, the surjectivity of $p$ does not necessarily imply the surjectivity of $f$. Here is an example. For $n\geq 1$, let $S_n$ be the zero-dimensional Stein space with~${\cO(S_n)=\C[\varepsilon]/(\varepsilon^n)}$. Let~$S$ be the disjoint union of the $(S_n)_{n\geq 1}$. Let $p:T:=S^{\red}\hookrightarrow S$ be the reduction of~$S$. As~$T$ is the zero locus of $\varepsilon$ in $S$, the surjective finite holomorphic map $p$ is globally finitely presented. Since the element $\varepsilon\in \cO(S)=\prod_{n\geq 1}\C[\varepsilon]/(\varepsilon^n)$ is not nilpotent, it does not vanish on all points of $\Spec(\cO(S))$. It follows that the closed immersion~$f:\Spec(\cO(T))=\Spec(\cO(S)/\varepsilon)\to\Spec(\cO(S))$ is not surjective.
\end{rem}

Let $S$ be a Stein space. To address the difficulty mentioned in Remark \ref{remnonsurj}, we say that a finite holomorphic map $p:T\to S$ is \textit{algebraically surjective} if the induced scheme morphism $\Spec(\cO(T))\to\Spec(\cO(S))$ is surjective.

Finally, the next lemma gives criteria for a finite holomorphic map between Stein spaces to be globally finitely presented.

\begin{lem}
\label{lemgfp}
Let $S$ be a Stein space.
\begin{enumerate}[label=(\roman*)] 
\item
\label{gfpi}
 If two finite holomorphic maps $p:T\to S$ and $p':T'\to S$ are globally finitely presented, then so is their fiber product $(p,p'):T\times_{S}T'\to S$.
\item
\label{gfpii}
If two finite holomorphic maps $p:T\to S$ and $q:T'\to T$ are globally finitely presented, then so is their composition $p\circ q$.
\item 
\label{gfpiii}
If $S$ is finite-dimensional, then any holomorphic map $p:T\to S$ that is finite flat of some degree $d\geq 1$ is globally finitely presented.
\end{enumerate}
\end{lem}

\begin{proof}
Assertions \ref{gfpi} and \ref{gfpii} follow from Proposition \ref{propadjeq} (since they become clear on the left side of the equivalence \eqref{adjeq}). Assertion \ref{gfpiii} holds because $p_*\cO_T$ is locally free of degree $d$, hence globally finitely presented by Proposition \ref{propmonoidaleq}.
\end{proof}

\subsection{Grothendieck topologies}
\label{parGroth}

We refer to  \cite[Definition~\href{https://stacks.math.columbia.edu/tag/03NH}{03NH}]{SP} for the definition of a site. Let $S$ be a Stein space. Let $X$ be a separated~$\cO(S)$\nobreakdash-scheme locally of finite presentation. As explained in \cite[\S 4.2]{Stein}, the analytification of \'etale schemes over~$X$ induces a morphism of sites $\varepsilon:(X^{\an})_{\cl}\to X_{\et}$ from the site of local isomorphisms of $X^{\an}$ (see \cite[XI, \S 4.0]{SGA43}) to the small \'etale site of $X$ (whose objects are the separated \'etale $X$-schemes of finite presentation, endowed with the \'etale topology). We define the analytification of an \'etale sheaf (or complex of sheaves) $\LL$ on $X$ to be $\LL^{\an}:=\varepsilon^*\LL$. As the topoi associated with $(X^{\an})_{\cl}$ and~$X^{\an}$ are equivalent (see \cite[III, Th\'eor\`eme~4.1]{SGA41}), pulling back by $\varepsilon$ induces, for all~$k\geq 0$, a comparison morphism
\begin{equation}
\label{defcompa}
H^k_{\et}(X,\LL)\to H^k(X^{\an},\LL^{\an}).
\end{equation}

If $X=\Spec(\cO(S))$, so $X^{\an}=S$, we consider the commutative diagram of sites
\begin{equation}
\label{sites}
\begin{aligned}
\xymatrix@C=2em@R=1em{
S_{\cl}\ar[d]&S_{\qc}\ar[r]\ar[l]\ar[d]&S_{\f}\ar[d]\\
\Spec(\cO(S))_{\et}&\Spec(\cO(S))_{\h}\ar[l]\ar[r]&\Spec(\cO(S))_{\f}.
}
\end{aligned}
\end{equation}
In \eqref{sites}, the objects of the site $\Spec(\cO(S))_{\f}$ are the finite $\cO(S)$-schemes of finite presentation, with coverings given by surjective morphisms. The site $S_{\f}$ has as objects all complex spaces over $S$ that are finite and globally finitely presented, with coverings given by holomorphic maps that are algebraically surjective in the sense defined in~\S\ref{paranal} (to verify that $S_{\f}$ is a site, use Lemma \ref{lemgfp} \ref{gfpi} and \ref{gfpii}).

The objects of the site $\Spec(\cO(S))_{\h}$ are the separated $\cO(S)$-schemes of finite presentation. It is endowed with Voevodsky's h topology (see \cite[Definition~\href{https://stacks.math.columbia.edu/tag/0ETS}{0ETS}]{SP}):
%(the definition of Voevodsky \cite[Definition 3.1.2]{Voehomo} was extended to the nonnoetherian setting by Rydh \cite[Definition 8.1]{Rydh}).
the coarsest Grothendieck topology for which Zariski coverings, as well as surjective proper morphisms of finite presentation, are coverings (see \cite[Lemmas~\href{https://stacks.math.columbia.edu/tag/0ETU}{0ETU}, \href{https://stacks.math.columbia.edu/tag/0ETV}{0ETV} and~\href{https://stacks.math.columbia.edu/tag/0ETW}{0ETW}]{SP}). The objects of $S_{\qc}$ are the Hausdorff and locally compact topological spaces over~$S$. It is endowed with the qc topology (see \cite[Definition~\href{https://stacks.math.columbia.edu/tag/09X0}{09X0}]{SP}). 

The morphisms in \eqref{sites} are the obvious ones, and~$S_{\qc}\to \Spec(\cO(S))_{\h}$ is well defined since analytifications of h coverings are qc coverings (see \cite[Lemma~\href{https://stacks.math.columbia.edu/tag/09X5}{09X5}]{SP}).

\subsection{The \'etale cohomology of a Stein algebra}
\label{parcompaS}

\begin{thm}
\label{thcompaS}
Let $S$ be a finite-dimensional Stein space. The morphisms
$$H^k_{\et}(\Spec(\cO(S)),\Z/m)\to H^k(S,\Z/m)$$
are isomorphisms for all $k\geq 0$ and all $m\geq 1$.
\end{thm}

\begin{proof}
We refer to \cite[Definition~\href{https://stacks.math.columbia.edu/tag/01G5}{01G5}]{SP} for the definition of a hypercovering of an object of a site. By \eqref{adjeq}, there is an equivalence of categories between hypercoverings $X_{\bullet}\to\Spec(\cO(S))$ in $\Spec(\cO(S))_{\f}$ and hypercoverings $T_{\bullet}\to S$ in~$S_{\f}$ (given by~$T_{\bullet}=X_{\bullet}^{\an}$ and $X_{\bullet}=\Spec(\cO(S_{\bullet}))$). Note that $X_{\bullet}\to\Spec(\cO(S))$ and~$T_{\bullet}\to S$ are hypercoverings for the h and qc topology respectively (see~\eqref{sites}). 

Take the colimits over all such hypercoverings (the transition maps being given by refinements of hypercoverings) of \v{C}ech-to-derived spectral sequences (for which see \cite[Lemma~\href{https://stacks.math.columbia.edu/tag/01GY}{01GY}]{SP}) for the h and qc topologies. This yields spectral sequences
\begin{alignat}{4}
\label{ssh}
&&E_2^{p,q}=\colim_{X_{\bullet}}\check{H}^p(X_{\bullet},\cH_{\h}^q(\Z/m))&\Rightarrow H^{p+q}_{\h}(\Spec(\cO(S)),\Z/m)\textrm{ \hspace{.5em}and } \\
\label{ssqc}
&&E_2^{p,q}=\colim_{T_{\bullet}}\check{H}^p(T_{\bullet},\cH_{\qc}^q(\Z/m))&\Rightarrow H^{p+q}_{\qc}(S,\Z/m),
\end{alignat}
where $\cH_{\h}^q(\Z/m)$ and $\cH_{\qc}^q(\Z/m)$ respectively are the presheaves $X\mapsto H^q_{\h}(X,\Z/m)$ and $T\mapsto H^q_{\qc}(T,\Z/m)$.
Using \cite[Lemma~\href{https://stacks.math.columbia.edu/tag/0EWH}{0EWH}]{SP} (which, in the excellent noetherian case, is due to Suslin and Voevodsky \cite[Corollary 10.10]{SV}) and \cite[Lemma~\href{https://stacks.math.columbia.edu/tag/09X4}{09X4}]{SP}, one can rewrite the spectral sequences \eqref{ssh} and \eqref{ssqc} as 
\begin{alignat}{4}
\label{sset}
&&E_2^{p,q}=\colim_{X_{\bullet}}\check{H}^p(X_{\bullet},\cH_{\et}^q(\Z/m))&\Rightarrow H^{p+q}_{\et}(\Spec(\cO(S)),\Z/m)\textrm{ \hspace{.5em}and } \\
\label{ss}
&&E_2^{p,q}=\colim_{T_{\bullet}}\check{H}^p(T_{\bullet},\cH^q(\Z/m))&\Rightarrow H^{p+q}(S,\Z/m),
\end{alignat}
where $\cH_{\et}^q(\Z/m)$ and $\cH^q(\Z/m)$ are the presheaves given by $X\mapsto H^q_{\et}(X,\Z/m)$ and~$T\mapsto H^q(T,\Z/m)$. In addition, there are natural morphisms from \eqref{ssh} to~\eqref{ssqc}, and hence from \eqref{sset} to \eqref{ss}, induced by vertical arrows in \eqref{sites}.

Fix $p\geq 0$ and $q>0$. For any morphism $X\to \Spec(\cO(S))$ in $\Spec(\cO(S))_{\f}$, any class in~$H^q_{\et}(X,\Z/m)$ can be killed after pullback by a covering $X'\to X$ in the site~$\Spec(\cO(S))_{\f}$ (this assertion can be deduced from \cite[Theorem 1.1]{Bhatt} using a limit argument based on \cite[Theorem~\href{https://stacks.math.columbia.edu/tag/09YQ}{09YQ}]{SP} to reduce to the excellent noetherian case). It therefore follows from the construction of refinements of hypercoverings given in \cite[Lemma~\href{https://stacks.math.columbia.edu/tag/01GJ}{01GJ}]{SP} that 
\begin{equation}
\label{et=0}
\colim_{X_{\bullet}}\check{H}^p(X_{\bullet},\cH_{\et}^q(\Z/m))=0.
\end{equation} 

Similarly, for any morphism $T\to S$ in $S_{\f}$, any class in~$H^q(T,\Z/m)$ can be killed after pullback by a finite flat holomorphic map $T'\to T$ of some degree $d\geq 1$, by Theorem~\ref{thkill}. As $T'\to T$ is a covering in the site $S_{\f}$ (by Lemma \ref{lemgfp} \ref{gfpiii} and Proposition \ref{propadjeq}), we deduce from the construction of refinements of hypercoverings given in \cite[Lemma~\href{https://stacks.math.columbia.edu/tag/01GJ}{01GJ}]{SP} that 
\begin{equation}
\label{=0}
\colim_{T_{\bullet}}\check{H}^p(T_{\bullet},\cH^q(\Z/m))=0.
\end{equation} 

For any morphism $X\to \Spec(\cO(S))$ in $\Spec(\cO(S))_{\f}$, Proposition~\ref{propadjeq} implies that closed and open subsets of $X$ are naturally in bijection with closed and open subsets of $X^{\an}$, and hence that the natural morphism $H^0_{\et}(X,\Z/m)\to H^0(X^{\an},\Z/m)$ is an isomorphism (the argument appears in \cite[Remark 3.11\,(iii)]{Steinsurface}). It follows that the natural morphism
\begin{equation}
\label{isoq=0}
\colim_{X_{\bullet}}\check{H}^p(X_{\bullet},\cH_{\et}^0(\Z/m))\to \colim_{T_{\bullet}}\check{H}^p(T_{\bullet},\cH^0(\Z/m))
\end{equation} 
is an isomorphism for all $p\geq 0$. 

Combining \eqref{et=0}, \eqref{=0} and \eqref{isoq=0} shows that the morphism between the spectral sequences \eqref{sset} and \eqref{ss} induces an isomorphism between their second pages, and hence also between their abutments. This completes the proof of the theorem.
\end{proof}

\subsection{Constructible coefficients}
\label{parconstructible}

\begin{prop}
\label{propdevissage}
Let $S$ be a finite-dimensional Stein space. Let $\LL$ be a constructible \'etale sheaf on $\Spec(\cO(S))$. Then the comparison morphisms
$$H^k_{\et}(\Spec(\cO(S)),\LL)\to H^k(S,\LL^{\an})$$
are isomorphisms for all $k\geq 0$.
\end{prop}

\begin{proof}
By \cite[Lemmas~\href{https://stacks.math.columbia.edu/tag/09Z7}{09Z7},~\href{https://stacks.math.columbia.edu/tag/095R}{095R} and~\href{https://stacks.math.columbia.edu/tag/03RX}{03RX}]{SP}, the sheaf $\LL$ has a resolution
\begin{equation}
\label{resolL}
0\to\LL\to\LL_0\to\LL_1\to\dots
\end{equation}
by constructible sheaves $\LL_p$ that are finite direct sums of sheaves of the form $\pi_*\M$, where~$\pi:X\to \Spec(\cO(S))$ is finite and of finite presentation and $\M$ is a constant constructible \'etale sheaf on $X$. In view of the spectral sequences
$$E_1^{p,q}=H^q_{\et}(\Spec(\cO(S)),\LL_p)\Rightarrow H^{p+q}_{\et}(\Spec(\cO(S)),\LL)$$
and
$$\hspace{.6em}E_1^{p,q}=H^q(S,\LL_p^{\an})\Rightarrow H^{p+q}(S,\LL^{\an}),$$
induced by  \eqref{resolL}, we can assume that $\LL$ is of the form $\pi_*\M$ with $\pi$ and $\M$ as above. 

Set $T:=X^{\an}$. As $\pi^{\an}:T\to S$ is finite, the complex space $T$ is Stein. In addition, one has $X=\Spec(\cO(T))$ by Proposition \ref{propadjeq}. The higher direct images of~$\pi$ and~$\pi^{\an}$ vanish (see \cite[Proposition~\href{https://stacks.math.columbia.edu/tag/03QP}{03QP}]{SP} and \cite[III, Theorem~6.2]{Iversen}). Using the Leray spectral sequences 
$$E_2^{p,q}=H^p_{\et}(\Spec(\cO(S)),\RR^q\pi_*\M)\Rightarrow H^{p+q}_{\et}(\Spec(\cO(T)),\M)$$
and
$$\hspace{.6em}E_2^{p,q}=H^p(S,\RR^q(\pi^{\an})_*\M^{\an})\Rightarrow H^{p+q}(T,\M^{\an}),$$one can thus reduce to the case where $\LL$ is constant constructible (after replacing~$S$ with~$T$ and $\LL$ with $\M$). We may finally assume that $\LL=\Z/m$ for some $m\geq 1$, in which case the statement to be proved is exactly Theorem~\ref{thcompaS}.
\end{proof}

\begin{thm}
\label{thproperpf}
Let $S$ be a finite-dimensional Stein space. Let $f: X\to\Spec(\cO(S))$ be a proper $\cO(S)$\nobreakdash-scheme of finite presentation. Let $\LL$ be a constructible \'etale sheaf on $X$. Then the comparison morphisms
$$H^k_{\et}(X,\LL)\to H^k(X^{\an},\LL^{\an})$$
are isomorphisms for all $k\geq 0$.
\end{thm}

\begin{proof}
Consider the Leray spectral sequences 
\begin{equation}
\label{Leray1}
E_2^{p,q}=H^p_{\et}(\Spec(\cO(S)), \RR^qf_*\LL)\Rightarrow H^{p+q}_{\et}(X,\LL)
\end{equation}
and
\begin{equation}
\label{Leray2}
E_2^{p,q}=H^p(S,\RR^q(f^{\an})_*\LL^{\an})\Rightarrow H^{p+q}(X^{\an},\LL^{\an}).
\end{equation}
The sheaves $\RR^qf_*\LL$ are constructible by \cite[Theorem~\href{https://stacks.math.columbia.edu/tag/0GL0}{0GL0}]{SP} and the natural morphisms $(\RR^qf_*\LL)^{\an}\to \RR^q(f^{\an})_*\LL^{\an}$ are isomorphisms by \cite[Theorem 4.9]{Stein}. The natural morphism from \eqref{Leray1} to \eqref{Leray2} is therefore an isomorphism on page $2$ by Proposition \ref{propdevissage}, hence an isomorphism between the abutments.
\end{proof}

\begin{rems}
\label{remscex}
(i)
Theorem \ref{thproperpf} does not hold in general if $S$ is not finite-dimensional (even for $k=1$, $X=\Spec(\cO(S))$ and $\LL=\Z/2$, see \cite[Remark 4.13\,(iii)]{Steinsurface}).

(ii) Theorem \ref{thproperpf} fails in general if $f$ is not proper (even for $S$ of dimension~$0$ and~$f$ an open immersion, with $k=0$ and $\LL=\Z/2$, see \cite[Remark 3.11\,(i)]{Steinsurface}).

(iii) Theorem \ref{thproperpf} does not hold in general if $f$ is assumed to be of finite type but not of finite presentation (even for $S$ of dimension $0$ and $f$ a closed immersion, with~$k=0$ and $\LL=\Z/2$, as the example of \cite[Remark 2.9]{Steinsurface} shows).

(iv) Theorem \ref{thproperpf} fails in general if the sheaf $\LL$ is not constructible (even for~$k=0$ and~${X=\Spec(\cO(S))}$ with $S$ of dimension $0$, see \cite[Remark 6.7\,(iii)]{Steinsurface}).
\end{rems}

\subsection{A \texorpdfstring{$G$}{G}-equivariant comparison theorem}
\label{parGeq}

Let $G:=\Gal(\C/\R)\simeq\Z/2$ be the group generated by complex conjugation. A $G$-\textit{equivariant complex space} is a complex space endowed, as a locally ringed space, with an action of $G$ such that the complex conjugation acts $\C$-antilinearly. It is said to be \textit{Stein} if so is the underlying complex space. For any $G$-equivariant Stein space $S$, there is an analytification functor associating with an $\cO(S)^G$-scheme locally of finite presentation~$X$ a $G$\nobreakdash-equivariant complex space $X^{\an}$ over $S$ (see \cite[\S 6.3]{Stein}).

Theorem \ref{thproperpf} readily extends to the $G$-equivariant setting. This extension is not used in this article, but is relevant for applications to real-analytic geometry (as in~\cite[\S 7]{Stein} or \cite[\S 10]{Steinsurface}).

\begin{thm}
\label{thproperpfG}
Let $S$ be a finite-dimensional $G$-equivariant Stein space. Let $X$ be a proper $\cO(S)^G$\nobreakdash-scheme of finite presentation. Let $\LL$ be a constructible \'etale sheaf on~$X$. Then the natural morphisms
$$H^k_{\et}(X,\LL)\to H^k_G(X^{\an},\LL^{\an})$$
are isomorphisms for all $k\geq 0$.
\end{thm}

\begin{proof}
The $G$-equivariant morphism of sites $\varepsilon:(X^{\an})_{\cl}\to (X_{\cO(S)})_{\et}$ induces a morphism between the Hochschild--Serre spectral sequences 
\begin{equation}
\label{HS1}
E_2^{p,q}=H^p(G, H^q_{\et}(X_{\cO(S)},\LL))\Rightarrow H^{p+q}_{\et}(X,\LL)
\end{equation}
and
\begin{equation}
\label{HS2}
E_2^{p,q}=H^p(G, H^q(X^{\an},\LL^{\an}))\Rightarrow H_G^{p+q}(X^{\an},\LL^{\an})
\end{equation}
(see \cite[Remark 10.9]{Scheiderer}). It induces an isomorphism between the second pages of~\eqref{HS1} and \eqref{HS2} by Theorem~\ref{thproperpf}, hence also between their abutments.
\end{proof}

\section{Application to the coniveau}
\label{secconiveau}

Theorem \ref{thintegral} below is an application of our main comparison theorem (Theorem~\ref{thproperpf}) to the coniveau of integral cohomology classes on Stein spaces.

\subsection{Coniveau of torsion cohomology classes, in algebraic geometry}
\label{parBCTV}

By way of preparation for the proof of Theorem \ref{thintegral}, we include the following refinement of Colliot-Th\'el\`ene and Voisin's \cite[Th\'eor\`eme 3.1]{CTV}. In its proof, the reduction to the smooth case uses Hoobler's trick \cite{Hoobler} (see also \cite[Proof of Theorem~7.8]{Kerz}).

\begin{thm}
\label{thCTV}
Let $V$ be a quasi-projective algebraic variety over $\C$. Fix $k\geq 0$ and a torsion class ${\alpha\in H^k(V^{\an},\Z)}$. Let $\Xi\subset V^{\an}$ be a countable subset. There exists a dense open subset $V'\subset V$ such that $\Xi\subset (V')^{\an}$ and $\alpha|_{(V')^{\an}}=0$ in $H^k((V')^{\an},\Z)$.
\end{thm}

\begin{proof}
Let $\overline{V}$ be a projective compactification of $V$. Let $D\subset \overline{V}$ be an ample divisor containing $\overline{V}\setminus V$ and such that $D^{\an}\cap \Xi=\varnothing$. After replacing $V$ with $\overline{V}\setminus D$ and $\alpha$ with~$\alpha|_{(\overline{V}\setminus D)^{\an}}$, we may assume that $V$ is affine.

Let $m\geq 1$ be such that $m\alpha=0$. Denote by $\partial_m$ the boundary maps induced by the short exact sequence $0\to\Z\xrightarrow{m}\Z\to\Z/m\to 0$. Choose a class $\beta\in H^{k-1}(V^{\an},\Z/m)$ such that $\partial_m(\beta)=\alpha$. Let $\tbeta\in H^{k-1}_{\et}(V,\Z/m)$ be the inverse image of $\beta$ by Artin's comparison isomorphism $H^{k-1}_{\et}(V,\Z/m)\isoto H^{k-1}(V^{\an},\Z/m)$.

Fix $N\geq 0$ and a closed embedding $V\subset \A^N_{\C}$. Let $i:V\hookrightarrow\whV$ be the henselization of~$\A^N_\C$ along $V$ (see \cite[Lemma~\href{https://stacks.math.columbia.edu/tag/0A02}{0A02}]{SP}). By Gabber's affine base change theorem (see \cite[Theorem~1]{Gabber}), there exists $\tgamma\in H^{k-1}_\et(\whV,\Z/m)$ with~$i^*\tgamma=\tbeta$. The scheme $\whV$ is a directed limit of affine \'etale $\A^N_{\C}$-schemes (see the proof of \cite[Lemma~\href{https://stacks.math.columbia.edu/tag/0A02}{0A02}]{SP}). By \cite[Theorem~\href{https://stacks.math.columbia.edu/tag/09YQ}{09YQ}]{SP}, there exist a smooth affine algebraic variety~$W$ over~$\C$, a morphism of~$\C$\nobreakdash-schemes~${f:\whV\to W}$, and a class~${\tdelta\in H^{k-1}_{\et}(W,\Z/m)}$ such that~$f^*\tdelta=\tgamma$. Let $\delta\in H^k(W^{\an},\Z/m)$ be the image of~$\tdelta$ by the comparison morphism $H^{k-1}_{\et}(W,\Z/m)\to H^{k-1}(W^{\an},\Z/m)$. Replacing~$V$ with $W$, the class $\alpha$ with~$\partial_m(\delta)$ and $\Xi$ with $(f\circ i)^{\an}(\Xi)$, we reduce to the case where $V$ is smooth.

By \cite[Th\'eor\`eme 3.1]{CTV}, there exists a dense open subset $V^+\subset V$ such that~$\alpha|_{(V^+)^{\an}}=0$.  Set $Z^+:=V\setminus V^+$, so $\alpha$ lifts to a class $\alpha^+\in H^k_{Z^+}(V,\Z)$. By the Bloch--Ogus--Gabber effacement theorem (see \cite[(4.2.3)]{BO}, \cite[Theorem~2.2.7]{CTHK} or \cite[Theorem 1.1]{Schreieder}, and Remark \ref{remBO} below) applied over $k=\C$ to the cohomology theory~$H^*=H^*((-)^{\an},\Z)$, there exists a dense open subset~${V'\subset V}$ with~${\Xi\subset (V')^{\an}}$ such that $\alpha^+$ (hence also $\alpha$) vanishes in $H^k((V')^{\an},\Z)$. 
\end{proof}

\begin{rem}
\label{remBO}
Keep the notation of the proof of Theorem \ref{thCTV}. The Bloch--Ogus--Gabber effacement theorem as stated in \cite{BO, CTHK, Schreieder} applies when~$\Xi$ is a single point (in \cite{BO}) or a finite set (in \cite{CTHK, Schreieder}). However, in the situation we consider (over $k=\C$ for~$H^*=H^*((-)^{\an},\Z)$), making very general choices in the proofs of \cite{BO, CTHK, Schreieder} allows them to go through when~$\Xi$ is countable. It is easier to check this for the proof of \cite[Theorem~2.2.7]{CTHK} given in \cite[\S\S3-4]{CTHK} (in \cite[\S3]{CTHK}, use that a countable intersection of dense open subsets in an algebraic variety over $\C$ contains $\C$-points).
%For Schreieder's proof, would need to go through the proof of Levine's moving lemma.
\end{rem}

\subsection{Coniveau of integral cohomology classes, up to a multiple}
\label{paruptomultiple}

As a first step towards Theorem \ref{thintegral}, we show that any integral cohomology class of degree~$\geq\nobreak2$ on a finite-dimensional Stein space has a multiple of coniveau~$\geq\nobreak1$, in a strong sense (see Propositions~\ref{propconiveauOka} and \ref{propconiveauuptomultiple}). Lemmas~\ref{lemodd} and \ref{lemeven} are used to prove Proposition~\ref{propconiveauOka} for odd and even degree classes respectively. Lemma \ref{lemeven} appears in~\cite[Proposition 3.1]{BV} but the proof there only applies to finite CW complexes. 

\begin{lem}
\label{lemodd}
Let $T$ be a finite-dimensional CW complex. Fix ${\alpha\in H^k(T,\Z)}$ for some $k\geq 1$ odd. Then there exist $m\geq 1$ and a continuous map~$f:T\to\bS^k$ such that~$m\alpha=f^*\omega$ (where $\omega\in H^k(\bS^k,\Z)$ denotes the canonical generator).
\end{lem}

\begin{proof}
Let $K(\Z,k)$ be an Eilenberg--MacLane CW complex with tautological class $\beta\in H^k(K(\Z,k),\Z)$. Let $K(\Z,k)_{\leq d}$ be the $d$-skeleton of $K(\Z,k)$. For $m\geq 1$, we let~${\mu_m:K(\Z,k)\to K(\Z,k)}$ be a cellular map (well defined up to homotopy) induced by multiplication by~$m$ on $\Z$, so that $(\mu_m)^*\beta=m\beta$.

Let $\nu:\bS^k\to K(\Z,k)$ be a continuous map classifying $\omega$ (so $\nu^*\beta=\omega$) and let $F_{\nu}$ be the homotopy fiber of $\nu$. Since $k$ is odd, Serre's results on the homotopy groups of spheres \cite[Proposition 3 p.\,494]{Serrehomologie} and the long exact sequence of homotopy groups in a fibration show that $\pi_i(F_{\nu})$ is finite for all $i$, and vanishes for $1\leq i\leq k$.

We claim that, for all $d\geq 0$, there exist an integer $m_d\geq 1$ and a continuous map~$g_d:K(\Z,k)_{\leq d}\to \bS^k$ such that $\nu\circ g_d=(\mu_{m_d})|_{K(\Z,k)_{\leq d}}$. The proof is by induction on~$d$. Assuming that $m_{d-1}$ and $g_{d-1}$ have been constructed, we need to extend the lift $g_{d-1}$ of $\nu$ to a lift $g_d:K(\Z,k)_{\leq d}\to\bS^k$, after maybe replacing $g_{d-1}$ with $g_{d-1}\circ\mu_{m'}$ for some $m'\geq 1$ (then set $m_d:=m'm_{d-1}$). The obstruction for doing so lives in $H^d(K(\Z,k),\pi_{d-1}(F_{\nu}))$ (see \cite[Theorem 34.2]{Steenrod}). As $\pi_{d-1}(F_{\nu})$ is finite, this obstruction class can be killed after pullback by $\mu_{m'}$ for $m'$ sufficiently divisible (the argument appears at the end of \cite[Proof of Lemma~2.2]{BV}).
%CHANGE numbering when published.

If $d$ is the dimension of $T$, one can find a continuous map $\varphi:T\to K(\Z,k)_{\leq d}$ classifying $\alpha$, so $\varphi^*(\beta|_{K(\Z,k)_{\leq d}})=\alpha$. Set $f:=g_d\circ\varphi$ and $m:=m_d$. To conclude the proof, one computes that
$$f^*\omega=f^*\nu^*\beta=\varphi^*g_d^*\nu^*\beta=\varphi^*(\mu_{m_d}|_{K(\Z,k)_{\leq d}})^*\beta=m\varphi^*(\beta|_{K(\Z,k)_{\leq d}})=m\alpha.\eqno\qedhere$$
\end{proof}

Let $\Gr(r,N)$ be the Grassmannian of quotients of dimension $r$ of~$\C^N$, and let~$\cE_{r,N}$ denote the tautological rank $r$ vector bundle on~$\Gr(r,N)$.

\begin{lem}
\label{lemeven}
Let $T$ be a finite-dimensional CW complex.  Fix $r\geq 1$ and classes $\alpha_i\in H^{2i}(T,\Z)$ for $1\leq i\leq r$. Then there exist integers $m\geq 1$ and $N\geq r$, and a continuous map $f:T\to \Gr(r,N)$ such that $m\alpha_i=f^*c_i(\cE_{r,N})$ for $1\leq i\leq r$.
\end{lem}

\begin{proof}
Let $\BGL_r(\C)$ be the classifying space of complex vector bundles of rank $r$, constructed as the colimit of the $\Gr(r, N)$ for $N\geq r$. Let $\cE_r$ be the universal vector bundle on $\BGL_r(\C)$ (so $\cE_r|_{\Gr(r, N)}=\cE_{r,N}$). Let $\lambda:\BGL_r(\C)\to\prod_{1\leq i\leq r}K(\Z,2i)$ be a continuous map classifying the Chern classes of $\cE_r$.

As $H^*(\BGL_r(\C),\Z)=\Z[c_1(\cE_r),\dots,c_r(\cE_r)]$, and in view of the computation of~$H^*(\prod_{1\leq i\leq r}K(\Z,2i),\Q)$ in \cite[Proposition 4 p.\,501]{Serrehomologie}, the map $\lambda$ induces an isomorphism on $H^*(-,\Q)$, hence on $H_*(-,\Q)$, hence on~$\pi_*(-)\otimes{\Q}$ by \cite[Th\'eor\`eme 3]{Serrehomotopie}. It also induces an isomorphism on~$\pi_i(-)$ for $i\leq 2$, by the Hurewicz theorem. Letting~$F_{\lambda}$ be the homotopy fiber of $\lambda$, the long exact sequence of homotopy groups of a fibration shows that $\pi_i(F_{\lambda})$ is finite, and vanishes for~${i\in\{1,2\}}$. 

Arguing as in the proof of Lemma \ref{lemodd} (using the map $\lambda$ instead of $\nu$), one finds an integer $m\geq 1$ and a continuous map $f:T\to\BGL_r(\C)$ such that~$m\alpha_i=f^*c_i(\cE_{r})$ for~$1\leq i\leq r$. Since $T$ is finite-dimensional and the inclusion $\Gr(r,N)\subset\BGL_r(\C)$ is a $(2N-2r+1)$-equivalence, the map $f$ factorizes through $\Gr(r, N)$ for $N$ big enough, which concludes the proof of the lemma.
\end{proof}

\begin{prop}
\label{propconiveauOka}
Let $S$ be a finite-dimensional Stein space. Fix ${\alpha\in H^k(S,\Z)}$ for some~$k\geq 1$. Set $c:=\lfloor\frac{k}{2}\rfloor$. There exist $m\geq 1$, an algebraic variety $V$ over~$\C$ that is homogeneous under a linear algebraic group $G$, a vector bundle $\cE$ of rank $c$ on $V$, a section~${\sigma\in H^0(V,\cE)}$ transverse to the zero section with zero locus $j:V'\hookrightarrow V$, a class $\zeta\in H^{k-2c}((V')^{\an},\Z)$, and a continuous map~${f:S\to V^{\an}}$ with $m\alpha=f^*(j^{\an})_*\zeta$.
\end{prop}

\begin{proof}
Recall that $S$ has the homotopy type of a finite-dimensional CW complex (see \cite[Korollar]{Hamm}), so the use of Lemmas \ref{lemodd} and \ref{lemeven} below is legitimate.

Assume first that $k=2c+1\geq 1$ is odd. Endow $V:=\A^{c+1}_\C\setminus\{(0,\dots,0)\}$ with its natural action of $G:=\GL_{c+1,\C}$. Set $\cE:=\cO_V^{\oplus c}$ and ${\sigma:=(x_1,\dots,x_c)\in\cO(V)^{\oplus c}}$ (so~${V'=\A^1_{\C}\setminus\{0\}}$). The manifolds $(V')^{\an}$ and $V^{\an}$ have the homotopy type of~$\bS^{1}$ and~$\bS^k$ respectively, and we let $\zeta\in H^1((V')^{\an},\Z)$ and $\omega\in H^k(V^{\an},\Z)$ be the canonical generators (so $\omega=(j^{\an})_*\zeta$). Lemma~\ref{lemodd} shows that there exists a continuous map~$f:S\to V^{\an}$ such that $m\alpha=f^*\omega=f^*(j^{\an})_*\zeta$ for some~${m\geq 1}$.

Assume now that $k=2c\geq 2$ is even. By Lemma~\ref{lemeven} (applied with~${r=c}$ and $(\alpha_1,\dots,\alpha_r)=(0,\dots,0,\alpha)$), there exist $m\geq 1$ and $N\geq r$, as well as a continuous map~${f:S\to \Gr(r,N)}$, such that $m\alpha=f^*c_r(\cE_{r,N})$. Let $V$ and $\cE$ be the algebraic variety and the algebraic vector bundle whose analytifications are $\Gr(r,N)$ and~$\cE_{r,N}$. Note that $V$ is homogeneous under the natural action of~$G:=\GL_{N,\C}$. As~$\cE$ is globally generated, the Bertini theorem shows that a general section~${\sigma\in H^0(V,\cE)}$ is transverse to the zero section. Let~${j:V'\hookrightarrow V}$ be its zero locus. Set~${\zeta:=1\in H^0((V')^{\an},\Z)}$, so~${(j^{\an})_*\zeta=c_r(\cE_{r,N})}$ and~${m\alpha=f^*(j^{\an})_*\zeta}$.
\end{proof}

We define \textit{Koszul-regular immersions} of schemes as in \cite[Definition~\href{https://stacks.math.columbia.edu/tag/063J}{063J}]{SP}. A Koszul-regular immersion has \textit{codimension} $c$ if its conormal sheaf (see \cite[Definition~\href{https://stacks.math.columbia.edu/tag/01R2}{01R2}]{SP}) is locally free of rank $c$.
%in this ref, Koszul-regular is the same as smooth-locally regular, see 0692.

\begin{prop}
\label{propconiveauuptomultiple}
Let $S$ be a connected Stein manifold. Let $\Xi\subset S$ be a countable subset. Fix ${\alpha\in H^k(S,\Z)}$ for some~$k\geq 2$. Set $c:=\lfloor\frac{k}{2}\rfloor$. There exist $m\geq 1$, a Koszul-regular closed immersion $\iota: X\hookrightarrow\Spec(\cO(S))$ of codimension $c$, and a cohomology class~${\beta\in H^{k-2c}(X^{\an},\Z)}$, such that $X^{\an}\cap\Xi=\varnothing$, such that $\iota^{\an}:X^{\an}\hookrightarrow S$ is the inclusion of a submanifold of codimension $c$, and such that $m\alpha=(\iota^{\an})_*\beta$.
\end{prop}

\begin{proof}
Let $(m,V, G,\cE,\sigma,j,V', \zeta,f)$ be as in Proposition \ref{propconiveauOka}. Since $V^{\an}$ is Oka (see \cite[Proposition~5.6.1]{Forstneric}), we may choose $f$ holomorphic (see Theorem ~\ref{thOka}). As~$V^{\an}$ is $G^{\an}$\nobreakdash-homogeneous, the~$\ci$ map $S\times G^{\an}\to V^{\an}$ given by~$(s,g)\mapsto g\cdot f(s)$ is a submersion. By the transversality theorem \cite[p.\,68]{GP}, the map~${f_g:s\mapsto g\cdot f(s)}$ is transverse to $j^{\an}$ for almost all $g\in G^{\an}$. Replacing~$f$ with~$f_g$ for a well-chosen~${g\in G^{\an}}$, we may assume that $f$ is transverse to $j^{\an}$ and that~$V'\cap f(\Xi)=\varnothing$. 

Set $\cF:=f^*\cE^{\an}$ and ${\tau:=f^*\sigma^{\an}\in H^0(S,\cF)}$. By transversality of~$f$ and~$j^{\an}$, the section~$\tau$ is transverse to the zero section and its zero locus $h:T\hookrightarrow S$ is a codimension~$c$ submanifold of $S$. Define ${\beta:=(f|_T)^*\zeta\in H^{k-2c}(T,\Z)}$. Again by transversality of $f$ and $j^{\an}$, one has ${m\alpha=f^*(j^{\an})_*\zeta=h_*(f|_T)^*\zeta=h_*\beta}$.

Applying the equivalence of categories of Proposition \ref{propmonoidaleq} to $\cF$ and $\tau$ yields a vector bundle $\cG$ of rank $c$ on $\Spec(\cO(S))$ and a section $\upsilon\in H^0(\Spec(\cO(S)),\cG)$. Let~$\iota: X\hookrightarrow\Spec(\cO(S))$ be the zero locus of $\upsilon$. The Koszul complex
\begin{equation}
\label{KoszulT}
0\to\extp^{c}\cF^\vee\xrightarrow{\tau}\dots\xrightarrow{\tau}\extp^2\cF^\vee\xrightarrow{\tau}\cF^\vee\xrightarrow{\tau}\cO_S\to\cO_T\to 0
\end{equation}
associated with $\tau$ is exact (use \cite[Lemma~\href{https://stacks.math.columbia.edu/tag/062F}{062F}]{SP} in the local rings of $S$). 
Applying the above equivalence of categories to~\eqref{KoszulT} gives rise to the Koszul complex
\begin{equation}
\label{KoszulX}
0\to\extp^{c}\cG^\vee\xrightarrow{\upsilon}\dots\xrightarrow{\upsilon}\extp^2\cG^\vee\xrightarrow{\upsilon}\cG^\vee\xrightarrow{\upsilon}\cO_{\Spec(\cO(S))}\to\cO_X\to 0
\end{equation}
associated with $\upsilon$. Since this equivalence is exact (see Proposition \ref{propmonoidaleq}), the complex~\eqref{KoszulX} is a resolution of $\cO_X$. This means that $\iota$ is a Koszul-regular immersion, of codimension $c$ because the conormal sheaf of $\iota$ is $\cG^\vee|_X$.

It follows from Proposition \ref{propmonoidaleq} that \eqref{KoszulT} is the analytification of \eqref{KoszulX}, hence that $T=X^{\an}$ and $h=\iota^{\an}$. This completes the proof of the proposition.
\end{proof}

\subsection{Applying the comparison theorem}
\label{parapplcompa}

Recall that we denote by $\partial_m$ the boundary maps induced by the short exact sequence $0\to\Z\xrightarrow{m}\Z\to\Z/m\to 0$. The proof of the next proposition is complicated by the fact that the comparison morphism $H^{k-1}_{\et}(U,\Z/m)\to H^{k-1}(U^{\an},\Z/m)$ that appears in its statement may not be an isomorphism.

\begin{prop}
\label{propBockstein}
Let $S$ be a connected Stein manifold. Let $\Xi\subset S$ be a countable subset. Fix~$k\geq 2$ and $\alpha\in H^k(S,\Z)$. There exist $m\geq 1$, a dense open subset~$U$ of~$\Spec(\cO(S))$, and a class~$\tdelta\in H^{k-1}_{\et}(U,\Z/m)$ with image $\delta\in H^{k-1}(U^{\an},\Z/m)$ by the comparison morphism $H^{k-1}_{\et}(U,\Z/m)\to H^{k-1}(U^{\an},\Z/m)$, such that $\Xi\subset U^{\an}$ and~${\partial_m(\delta)=\alpha|_{U^{\an}}}$ in~$H^k(U^{\an},\Z)$.
\end{prop}

\begin{proof}
Set $c:=\lfloor\frac{k}{2}\rfloor$. Let $m$, as well as $\iota:X\to \Spec(\cO(S))$ and $\beta\in H^{k-2c}(X^{\an},\Z)$ be as in Proposition~\ref{propconiveauuptomultiple}. Set $U:=\Spec(\cO(S))\setminus X$, so $U^{\an}=S\setminus X^{\an}$.

Let $\Cl^{\Z}_{\iota^{\an}}:\Z[-2c]\isoto(\iota^{\an})^!\Z$ be the Thom isomorphism in $D^+(X^{\an},\Z)$. After pushforward by $\iota^{\an}$ and composition with the natural morphism ${(\iota^{\an})_*(\iota^{\an})^! \Z\to\Z}$, it gives rise to the Gysin morphism $\Gys^\Z_{\iota^{\an}}:(\iota^{\an})_*\Z[-2c]\to\Z$ in $D^+(S,\Z)$. Let~${\K\in D^+(S,\Z)}$ be defined by the distinguished triangle
\begin{equation}
\label{triangleanZ}
\Z\oplus(\iota^{\an})_*\Z[-2c]\xrightarrow{(m,\Gys^\Z_{\iota^{\an}})}\Z\to\K\xrightarrow{+1}.
\end{equation}
As $m\alpha=(\iota^{\an})_*\beta$ (see Proposition~\ref{propconiveauuptomultiple}), the class $(\alpha,-\beta)\in H^k(S, \Z\oplus(\iota^{\an})_*\Z[-2c])$ lifts, in the long exact sequence of cohomology of \eqref{triangleanZ}, to a class $\gamma\in H^{k-1}(S,\K)$. Since the restriction of \eqref{triangleanZ} to $U^{\an}$ reads
\begin{equation}
\label{ZZZm}
\Z\xrightarrow{m}\Z\to\Z/m\xrightarrow{+1},
\end{equation}
one has $\K|_{U^{\an}}=\Z/m$, and $\delta:=\gamma|_{U^{\an}}\in H^{k-1}(U^{\an},\Z/m)$ satisfies ${\partial_m(\delta)=\alpha|_{U^{\an}}}$. 
It remains to show that $\delta$ lifts to a class $\tilde{\delta}\in H^{k-1}_{\et}(U,\Z/m)$. 

Let $\Cl_{\iota^{\an}}:\Z/m[-2c]\isoto(\iota^{\an})^!\Z/m$ be the modulo $m$ Thom isomorphism in~$D^+(X^{\an},\Z/m)$, and denote by $\Gys_{\iota^{\an}}:(\iota^{\an})_*\Z/m[-2c]\to\Z/m$ the induced Gysin morphism in~$D^+(S,\Z/m)$. 
Consider a morphism
\begin{equation}
\label{morphtriangles}
\begin{aligned}
\xymatrix@C=3.5em@R=1em{
\Z\oplus(\iota^{\an})_*\Z[-2c]\ar^{(0,\rho)}[d]\ar^{\hspace{3.2em}(m,\Gys^\Z_{\iota^{\an}})}[r]&\Z\ar[r]\ar^{\rho}[d]&\K\ar^{\rho_\K}[d]\ar^{+1}[r]&\\
(\iota^{\an})_*\Z/m[-2c]\ar^{\hspace{2em}\Gys_{\iota^{\an}}}[r]&\Z/m\ar[r]&\overline{\K}\ar^{+1}[r]&
}
\end{aligned}
\end{equation}
of distinguished triangles in $D^+(S,\Z)$ whose upper row is \eqref{triangleanZ}, where $\rho$ denotes reduction modulo $m$ morphisms, where $\overline{\K}$ is defined by the bottom row of \eqref{morphtriangles}, and where $\rho_\K:\K\to\overline{\K}$ is constructed using axiom (TR3) of triangulated categories.
Set~${\bar{\gamma}:=\rho_\K(\gamma)\in H^{k-1}(S,\overline{\K})}$. The restriction of \eqref{morphtriangles} to $U^{\an}$ reads
\begin{equation}
\label{morphtrianglesU}
\begin{aligned}
\xymatrix@C=3.5em@R=1em{
\Z\ar^{}[d]\ar^{m}[r]&\Z\ar^{\rho}[r]\ar^{\rho}[d]&\Z/m\ar^{\Id}[d]\ar^{+1}[r]&\\
0\ar^{}[r]&\Z/m\ar^{\Id}[r]&\Z/m\ar^{+1}[r]&.
}
\end{aligned}
\end{equation}
It follows from \eqref{morphtrianglesU} that $\bar{\gamma}|_{U^{\an}}=\gamma|_{U^{\an}}=\delta$ in $H^{k-1}(U^{\an},\Z/m)$.

Let $\Cl_{\iota}:\Z/m[-2c]\to \iota^! \Z/m$ be the morphism in $D^+(X_\et,\Z/m)$ defined in \cite[D\'efinition 2.3.1]{Riou} (we consider a shift of the morphism in \emph{loc}.\kern3pt \emph{cit}., we identify~$\Z/m$ with the Tate twist $\Z/m(1)$ by sending $1$ to $e^{\frac{2i\pi}{m}}$, and we note that the regular immersions of \cite{Riou} are our Koszul-regular immersions). Consider the induced Gysin morphism $\Gys_{\iota}:\iota_*\Z/m[-2c]\to\Z/m$ in~$D^+(\Spec(\cO(S))_{\et},\Z/m)$. Let~${\LL\in D^+(\Spec(\cO(S))_{\et},\Z/m)}$ be defined by the distinguished triangle
\begin{equation}
\label{trianglealg}
\iota_*\Z/m[-2c]\xrightarrow{\Gys_{\iota}}\Z/m\to \LL\xrightarrow{+1}.
\end{equation}

\begin{lem}
\label{lemThomanal}
The bottom row of \eqref{morphtriangles} is the analytification of \eqref{trianglealg} in the sense of \S\ref{parGroth}. In particular, one has $\LL^{\an}=\overline{\K}$.
\end{lem}

\begin{proof}
It suffices to verify that the analytification of~$\Cl_{\iota}$ (which we view as an element of~$H^{2c}_{\et,X}(\Spec(\cO(S)),\Z/m)$) is equal to~$\Cl_{\iota^{\an}}$ (viewed in~$H^{2c}_{X^{\an}}(S,\Z/m)$). By uniqueness of Thom classes, we may work locally at some point ${s\in X^{\an}\subset S}$. 

Let~$U\subset\Spec(\cO(S))$ be an affine open neighborhood of $s$ in which $X$ is defined by a Koszul-regular sequence $(f_j)_{1\leq j\leq c}$ in $\cO(U)$. For $1\leq j\leq c$, consider the codimension $1$ Koszul-regular immersion $\iota_j:D_j:=\{f_j=0\}\hookrightarrow U$. Let~${\Omega\subset U^{\an}}$  be an open neighborhood of $s$ such that the $(D_j^{\an}\cap\Omega)_{1\leq j\leq c}$ are codimension $1$ complex submanifolds of $\Omega$ intersecting transversally along $X^{\an}\cap\Omega$.

Let $\Cl_{\iota_j}\in H^2_{\et,D_j}(U,\Z/m)$ be the class of $\iota_j$ (see \cite[D\'efinition 2.3.1]{Riou}) and let~$\Cl_{\iota_j^{\an}|_{\Omega}}\in H^2_{D_j^{\an}\cap\Omega}(\Omega,\Z/m)$ be the modulo $m$ Thom class of~${\iota_j^{\an}|_{\Omega}:D_j^{\an}\cap\Omega\hookrightarrow\Omega}$.
Then $\Cl_{\iota}|_U$ is the product of the $(\Cl_{\iota_j})_{1\leq j\leq c}$ by \cite[Remarque~2.3.6]{Riou}, and a topological computation shows that $\Cl_{\iota^{\an}}|_{\Omega}$ is the product of the $(\Cl_{\iota_j^{\an}|_{\Omega}})_{1\leq j\leq c}$. It remains to show that the restriction to $\Omega$ of the analytification of $\Cl_{\iota_j}$ is $\Cl_{\iota_j^{\an}|_{\Omega}}$. This follows from the concrete description of $\Cl_{\iota_j}$ given in \mbox{\cite[\S 2.1]{Riou}} and from the similar description of Thom classes of codimension $1$ complex submanifolds.
\end{proof}

We resume the proof of Proposition \ref{propBockstein}. As the closed immersion $\iota$ is Koszul-regular, it is of finite presentation (to see it, use local Koszul resolutions), so the \'etale sheaf $\iota_*\Z/m$ is constructible (see \cite[Lemma~\href{https://stacks.math.columbia.edu/tag/095R}{095R}]{SP}). In view of \eqref{trianglealg}, Theorem~\ref{thproperpf} (applied on~$X=\Spec(\cO(S))$ to the constructible sheaves $\iota_*\Z/m$ and $\Z/m$) and the five lemma imply that the comparison morphism
$$H^{k-1}_{\et}(\Spec(\cO(S)),\LL)\to H^{k-1}(S,\LL^{\an})=H^{k-1}(S,\overline{\K})$$
(where we used Lemma \ref{lemThomanal}) is an isomorphism. Let~$\tilde{\gamma}\in H^{k-1}_{\et}(\Spec(\cO(S)),\LL)$ be the preimage of $\bar{\gamma}$ by this isomorphism. The restriction of \eqref{trianglealg} to $U$ reads 
\begin{equation}
\label{trianglealgU}
0\xrightarrow{}\Z/m\xrightarrow{\Id}\Z/m\xrightarrow{+1}.
\end{equation}
Define $\tdelta:=\tgamma|_U\in H^{k-1}_{\et}(U,\Z/m)$. Then the image of $\tdelta$ by the comparison morphism $H^{k-1}_{\et}(U,\Z/m)\to H^{k-1}(U^{\an},\Z/m)$ is $\bar{\gamma}|_{U^{\an}}=\delta$. The proof is complete.
\end{proof}

\subsection{Coniveau of integral cohomology classes}
\label{parconiveau}

We now state and prove the main result of this section.

\begin{thm}
\label{thintegral}
Let $S$ be a finite-dimensional Stein space. Let $\Xi\subset S$ be a countable subset. For all~$k\geq 2$ and~$\alpha\in H^k(S,\Z)$, there exists a nowhere dense closed analytic subset~$Z\subset S$ such that $Z\cap\Xi=\varnothing$ and $\alpha|_{S\setminus Z}=0$ in~$H^k(S\setminus Z,\Z)$. 
\end{thm}

\begin{proof}
We first reduce to the case where $S$ is a manifold. By \cite[Lemma~2.1]{Stein}, there exist a Stein manifold~$S'$, a holomorphic map $i:S\to S'$ and a continuous retraction~${r:S'\to S}$ of $i$. Set $\alpha':=r^*\alpha\in H^k(S',\Z)$. Let~$\Xi'\subset S'$ be the union of~$i(\Xi)$ and of one point in the image by $i$ of each irreducible component of $S$. By the manifold case of Theorem \ref{thintegral}, there is a nowhere dense closed analytic subset~$Z'\subset S'$ with~${Z'\cap\Xi'=\varnothing}$ and $\alpha'|_{S'\setminus Z'}=0$ in~$H^k(S'\setminus Z',\Z)$. To conclude, define~${Z:=i^{-1}(Z')}$.

Assume now that~$S$ is a Stein manifold. We may furthermore assume that~$S$ is connected.  Let~$(m, U, \tdelta, \delta)$ be as in Proposition~\ref{propBockstein}. After shrinking $U$, we may assume that $U$ is affine. Writing $\cO(U)$ as the directed colimit of its sub\nobreakdash-$\C$\nobreakdash-algebras of finite type and applying \cite[Theorem~\href{https://stacks.math.columbia.edu/tag/09YQ}{09YQ}]{SP}, one finds an integral affine algebraic variety~$V$ over~$\C$, a morphism of $\C$\nobreakdash-schemes~${g:U\to V}$, and a class~${\tvarepsilon\in H^{k-1}_{\et}(V,\Z/m)}$ with~$g^*\tvarepsilon=\tdelta$. Let $g^{\an}:U^{\an}\to V^{\an}$ be the holomorphic map given by the universal property of~$V^{\an}$ (see \cite[XII, Th\'eor\`eme~1.1]{SGA1}).

Let $\varepsilon\in H^{k-1}(V^{\an},\Z/m)$ be the image of the class $\tvarepsilon$ by the comparison morphism~$H^{k-1}_{\et}(V,\Z/m)\to H^{k-1}(V^{\an},\Z/m)$, so $(g^{\an})^*\varepsilon=\delta$. By Theorem \ref{thCTV}, there exists a dense open subset $V'\subset V$ with $g^{\an}(\Xi)\subset(V')^{\an}$ and $\partial_m(\varepsilon)|_{(V')^{\an}}=0$ in~$H^k((V')^{\an},\Z)$. Set $U':=g^{-1}(V')$, so $(U')^{\an}=(g^{\an})^{-1}((V')^{\an})$. One computes
$$\alpha|_{(U')^{\an}}=\partial_m(\delta)|_{(U')^{\an}}=(g^{\an}|_{(U')^{\an}})^*(\partial_m(\varepsilon)|_{(V')^{\an}})=0\textrm{ in } H^k((U')^{\an},\Z).$$

Note that $U'$ is a dense open subset of $U$, hence also of $\Spec(\cO(S))$. Let~${a\in\cO(S)}$ be a nonzero element that vanishes on the complement of $U'$ in $\Spec(\cO(S))$, but that is nonzero at all points of $\Xi$. Define~$Z:=\{a=0\}\subset S$. As $S\setminus Z\subset (U')^{\an}$, the class~$\alpha|_{S\setminus Z}\in H^k(S\setminus Z,\Z)$ vanishes, which concludes the proof of the theorem.
\end{proof}

\begin{rem}
The analogue of Theorem \ref{thintegral} with finite coefficients does not hold in general, even for $k=2$. Indeed, if $S$ is a Stein manifold and $\alpha\in H^2(S,\Z/m)$ vanishes on the complement of a nowhere dense closed analytic subset $Z$ of $S$, then~$\alpha$ is the class of a divisor of $S$ supported on $Z$, and hence is the reduction modulo~$m$ of a class in $H^2(S,\Z)$. As a consequence, any $\alpha$ that does not lift to an integral cohomology class cannot vanish on the complement of such a $Z$.
%also in higher degree (for sure degree 3 and m=2) using Steenrod operations.
%with coefficients Q, fails for k=2 (for classes not coming from H^*(-,Z)_Q, but seems OK for higher degree (at least odd) by constructing over bigger and bigger compact subsets of S morphisms to C^l-0 that are compatible upto raising the last coordinate to a power (and up to small perturbation).
\end{rem}

\subsection{Cohomology and unramified cohomology of Stein spaces}
\label{parunramified}

Let $S$ be a Stein space. For $k\geq 0$, define $\cH^k_S(\Z)$ to be the Zariski sheaf on $\Spec(\cO(S))$ associated with the presheaf $U\mapsto H^k(U^{\an},\Z)$. Equivalently, one has ${\cH^k_S(\Z)=\RR^ki_*\Z}$, where~${i:S\to\Spec(\cO(S))}$ is the canonical continuous morphism from $S$ (with its usual topology) to $\Spec(\cO(S))$ (with the Zariski topology). We define the \textit{unramified cohomology groups} of $S$ with integral coefficients to be $$H^k_{\nr}(S,\Z):=H^0(\Spec(\cO(S)),\cH^k_S(\Z)).$$

\begin{thm}
\label{thunramified}
Let $S$ be a finite-dimensional Stein space. For all $k\geq 2$, the natural morphism $H^k(S,\Z)\to H^k_{\nr}(S,\Z)$ vanishes identically. 
\end{thm}

\begin{proof}
Fix $\alpha\in H^k(S,\Z)$ and $x\in\Spec(\cO(S))$. Let~${\kp\subset\cO(S)}$ be the prime ideal associated with $x$. By Lemma \ref{lemideals} below, there exists a discrete (hence countable) subset~$\Xi\subset S$ such that any element of $\kp$ vanishes at some point of~$\Xi$. By Theorem~\ref{thintegral}, there exists a closed analytic subset $Z\subset S$ such that $Z\cap\Xi=\varnothing$ and~$\alpha|_{S\setminus Z}=0$ in~$H^k(S\setminus Z,\Z)$. Choose $a\in \cO(S)$ vanishing on $Z$ but not on any point of $\Xi$ (so~$a\notin\kp$). Define~$U:=\Spec(\cO(S)[\frac{1}{a}])\subset\Spec(\cO(S))$, so~$U^{\an}=S\setminus\{a=0\}$. One has~$x\in U$ (because~$a\notin\kp$) and $\alpha|_{U^{\an}}=0$ (because~$U^{\an}\subset S\setminus Z$). Since $x$ was arbitrary, this shows that the image of $\alpha$ in $H^k_{\nr}(S,\Z)$ vanishes.
\end{proof}

\begin{lem}
\label{lemideals}
Let $S$ be a finite-dimensional Stein space. Let $I\subsetneq\cO(S)$ be an ideal. There exists a discrete subset $\Xi\subset S$ such that any $a\in I$ vanishes at a point of~$\Xi$. 
 \end{lem}

\begin{proof}
Choose a finite collection~$(a_1,\dots, a_r)$ of elements of $I$ such that the dimension of $T:=\{a_1=\dots=a_r=0\}\subset S$ is minimal. Since $S$ is Stein, the exact sequence~$\cO_S^{\oplus r}\xrightarrow{(a_1,\dots,a_r)}\cO_S\to\cO_T\to 0$ remains exact after taking global sections. As $I\neq\cO(S)$, we deduce that $T$ is nonempty. 

Let $\Xi\subset S$ be a discrete subset containing at least one point from each irreducible component of~$T$. Suppose for contradiction that $a\in I$ does not vanish at any point of $\Xi$. Then~$T':=\{a=a_1=\dots=a_r=0\}\subset T$ contains no irreducible component of $T$, so $\dim(T')<\dim(T)$. This contradicts our choice of $T$.
\end{proof}

\begin{rems}
\label{remsunramified}
(i)
Let $S$ be a finite-dimensional Stein space and fix $k\geq 2$. We do not know whether $H^k_{\nr}(S,\Z)$ vanishes (this would strengthen Theorem \ref{thunramified}). 
%even after tensorization with Q!

(ii) 
We do not know whether $\cH_S^k(\Z)$ is torsion-free for all finite-dimen\-sional Stein spaces and all $k\geq 0$. This would be a Stein analogue of \cite[Th\'eor\`eme 3.1]{CTV}.

(iii) The Leray spectral sequence of the morphism $i:S\to\Spec(\cO(S))$ reads
\begin{equation}
\label{coniveauss}
H^p(\Spec(\cO(S)),\cH^q_S(\Z))\Rightarrow H^{p+q}(S,\Z).
\end{equation}
It is an analogue in Stein geometry of the coniveau spectral sequence of \cite{BO,CTV}. Theorem~\ref{thunramified} suggests that the behaviour of \eqref{coniveauss} may be very different from that of the algebraic coniveau spectral sequence.
\end{rems}

\section{Cohomology of Stein spaces from algebraic varieties}
\label{secweight}

\subsection{Finite coefficients}
\label{paralgebraization}

By the following theorem, singular cohomology classes with finite coefficients on a finite-dimensional Stein space have an algebraic origin.

\begin{thm}
\label{thmcolim}
Let $S$ be a finite-dimensional Stein space. Fix~$k\geq 0$ and $m\geq 1$. Letting $f:S\to V^{\an}$ run over all holomorphic maps from $S$ to the analytification of some quasi-projective algebraic variety $V$ over $\C$ (with transition maps induced by analytifications of morphisms of algebraic varieties) gives rise to an isomorphism
\begin{equation}
\label{colim}
\underset{f:S\to V^{\an}}{\colim }H^k(V^{\an},\Z/m)\isoto H^k(S,\Z/m).
\end{equation}
In addition, in \eqref{colim}, one can restrict to those $f:S\to V^{\an}$ with $V$ affine.
\end{thm}

\begin{proof}
Fix $f:S\to V^{\an}$ as in the theorem. By Jouanolou's trick (see \cite[Lemme~1.5]{Jouanolou}), there exist a vector bundle $E$ on~$V$ and an $E$-torsor $\pi:V'\to V$ with~$V'$ affine. By Oka theory (see Theorem \ref{thOka}), the map~$f$ lifts to a holomorphic map~$f':S\to (V')^{\an}$. This shows that the $f:S\to V^{\an}$ with $V$ affine form a cofinal family, so we may and will henceforth restrict our attention to such~$f$.

If $V$ is an affine algebraic variety over $\C$, there is a natural bijection between holomorphic maps~${f:S\to V^{\an}}$ and morphisms~$g:\Spec(\cO(S))\to V$ (if $V$ is defined by the vanishing of some~${F_i\in \C[X_1,\dots,X_N]}$ in $\A^N_{\C}$, both are given by $N$ elements of~$\cO(S)$ annihilating the $F_i$). We deduce a commutative diagram
\begin{equation}
\label{diagcolim}
\begin{aligned}
\xymatrix@C=2em@R=1em{
\underset{g:\Spec(\cO(S))\to V}{\colim}H^k_{\et}(V,\Z/m)\ar[r]\ar[d]&H^k_{\et}(\Spec(\cO(S)),\Z/m)\ar[d]\\
\underset{f:S\to V^{\an}}{\colim}H^k(V^{\an},\Z/m)\ar[r]&H^k(S,\Z/m).
}
\end{aligned}
\end{equation}
In \eqref{diagcolim}, the upper horizontal arrow is an isomorphism by  \cite[Theorem~\href{https://stacks.math.columbia.edu/tag/09YQ}{09YQ}]{SP}. So are the vertical arrows by Artin's comparison theorem \cite[XVI, Th\'eor\`eme~4.1]{SGA43} and Theorem \ref{thcompaS} respectively. The theorem follows.
\end{proof}

\begin{cor}
 \label{corcolim}
Let $S$ be a finite-dimensional Stein space. Fix $\alpha\in H^k(S,\Z/m)$ for some $k\geq 0$ and $m\geq 1$. There exist an affine algebraic variety $V$ over $\C$, a holomorphic map $f:S\to V^{\an}$, and a class $\beta\in H^k(V^{\an},\Z/m)$ such that $\alpha=f^*\beta$.
\end{cor}

\subsection{Integral coefficients}

Our next result is an integral variant of Theorem~\ref{thmcolim}.

\begin{thm}
\label{thmcolimZ}
Let $S$ be a finite-dimensional Stein space. For any~$k\geq 1$, letting $f:S\to V^{\an}$ run over all holomorphic maps from $S$ to the analytification of some quasi-projective algebraic variety $V$ over $\C$ (with transition maps induced by analytifications of morphisms of algebraic varieties) gives rise to a morphism
\begin{equation}
\label{colimZ}
\underset{f:S\to V^{\an}}{\colim }H^k(V^{\an},\Z)\to H^k(S,\Z)
\end{equation}
which is surjective, and bijective in restriction to the torsion subgroups.

The statement remains true if one restricts to those $f:S\to V^{\an}$ with $V$ affine.
\end{thm}

\begin{proof}
By Jouanolou's trick applied as in the first paragraph of the proof of Theorem~\ref{colim}, it suffices to prove the first statement (in which $V$ may not be affine).

Let $f:S\to V^{\an}$ be as in the theorem. Consider the long exact sequences of cohomology of $S$ and $V^{\an}$ associated with~$0\to\Z\xrightarrow{m}\Z\to \Z/m\to 0$. Take the colimit of these long exact sequences over all $m\geq 1$, with transition maps induced by the morphisms of short exact sequences
\begin{equation}
\label{diagmm'}
\begin{aligned}
\xymatrix@C=2em@R=.75em{
0\ar[r]&\Z\ar^m[r]\ar@{=}[d]&\Z\ar[r]\ar^{m'}[d]&\Z/m\ar[r]\ar[d]&0\\
0\ar[r]&\Z\ar^{mm'}[r]&\Z\ar[r]&\Z/mm'\ar[r]&0,
}
\end{aligned}
\end{equation}
as well as over all $f:S\to V^{\an}$. This yields a commutative exact diagram
\begin{equation}
\label{diagcolimcolim}
\begin{aligned}
\xymatrix@C=2em@R=.75em{
\underset{f:S\to V^{\an}}{\colim}\underset{m}{\colim}\,
H^{k-1}(V^{\an},\Z)\ar@{->>}[r]\ar[d]&\underset{m}{\colim}\,H^{k-1}(S,\Z)\ar[d]\\
\underset{f:S\to V^{\an}}{\colim}\underset{m}{\colim}\,
H^{k-1}(V^{\an},\Z/m)\ar^{\hspace{2em}\sim}[r]\ar[d]&\underset{m}{\colim}\,H^{k-1}(S,\Z/m)\ar[d]\\
\underset{f:S\to V^{\an}}{\colim}H^k(V^{\an},\Z)\ar[r]\ar[d]&H^k(S,\Z)\ar[d]\\
\underset{f:S\to V^{\an}}{\colim}\underset{m}{\colim}\,
H^k(V^{\an},\Z)\ar@{->>}[r]\ar[d]&\underset{m}{\colim}\,H^k(S,\Z)\ar[d]\\
\underset{f:S\to V^{\an}}{\colim}\underset{m}{\colim}\,
H^k(V^{\an},\Z/m)\ar^{\hspace{2em}\sim}[r]&\underset{m}{\colim}\,H^k(S,\Z/m).
}
\end{aligned}
\end{equation}
Note that, when we write $\underset{f:S\to V^{\an}}{\colim}\underset{m}{\colim}$, the two colimits commute, and when we write $\underset{m}{\colim}\,H^*(-,\Z)$, the transition morphisms are given by multiplication maps. The second and fifth rows of \eqref{diagcolimcolim} are isomorphisms by Theorem~\ref{thmcolim}, and the first and fourth rows are surjective as a consequence of Proposition~\ref{propconiveauOka} (except maybe the first row when $k=1$). The theorem is now proved by a diagram chase in~\eqref{diagcolimcolim}, taking into account that the torsion subgroups of~$H^k(V^{\an},\Z)$ and~$H^k(S,\Z)$ are exactly the images of $\underset{m}{\colim}\,H^{k-1}(V^{\an},\Z/m)$ and~$\underset{m}{\colim}\,H^{k-1}(S,\Z/m)$ (except for the assertion concerning torsion subgroups when $k=1$, which is trivial because~$H^1(V^{\an},\Z)$ and~$H^1(S,\Z)$ have no torsion).
\end{proof}

\begin{cor}
 \label{corcolimZ}
Let $S$ be a finite-dimensional Stein space. Fix $\alpha\in H^k(S,\Z)$ for some $k\geq 1$. There exist an affine algebraic variety $V$ over $\C$, a holomorphic map~ $f:S\to V^{\an}$, and a class $\beta\in H^k(V^{\an},\Z)$ such that $\alpha=f^*\beta$.
\end{cor}

\begin{rems}
\label{remcolim}
(i) The morphism \eqref{colimZ} is not surjective for $k=0$ in general (a class~$\alpha\in H^0(S,\Z)$ taking infinitely many distinct values is not in its image).

(ii) Let us show that \eqref{colimZ} is not injective in general. Serre constructed two quasi-projective algebraic varieties~$V$ and~$V'$ over $\C$ such that $V^{\an}$ and~$(V')^{\an}$ are both biholomorphic to~$(\C^\times)^2$, but such that~$H^1(V^{\an},\Q)$ and $H^1((V')^{\an},\Q)$ are pure of weights~$1$ and $2$ in the sense of Deligne \cite{Hodge2, Hodge3} (see \cite[VI, Example~3.2]{Hartshorneample} for the construction of~$V$ and~$V'$). As morphisms of rational Hodge structures are strict with respect to the weight filtration (see \cite[Th\'eor\`eme~2.3.5\,(iii)]{Hodge2}), there cannot exist an algebraic variety~$W$ over $\C$ and morphisms~$h:W\to V$ and~$h':W\to V'$ such that~$h^*(H^1(V^{\an},\Q))$ and~$(h')^*(H^1((V')^{\an},\Q))$ are equal and nonzero. This implies that \eqref{colimZ} is not injective for $S=(\C^\times)^2$.
 \end{rems}

\subsection{The Stein weight filtration}
\label{parGS}

Inspired by Deligne's mixed Hodge theory, Gillet and Soul\'e defined a weight filtration $W_{\bullet}$ on the compactly supported singular cohomology of algebraic varieties over $\C$, with arbitrary coefficients (see~\mbox{\cite[\S 3.1.2]{GS}}). Their construction was extended to cohomology without support by Cirici and Guill\'en \mbox{\cite[\S6]{CG}} (this is attributed to~\cite{GNA} in the last lines of~\cite[\S 3.1.2]{GS}).

Let $S$ be a finite-dimensional Stein space. Fix $k\geq 0$ and $m\geq 1$. For~${j\in\Z}$, we define~$W^{\St}_jH^k(S,\Z/m)$ to be the subgroup of~$H^k(S,\Z/m)$ spanned by the~$f^*\alpha$, where~${f:S\to V^{\an}}$ is a holomorphic map to the analytification of some quasi-projective algebraic variety~$V$ over~$\C$, and where~$\alpha\in W_jH^k(V^{\an},\Z/m)$. In view of Theorem~\ref{thmcolim}, one can restrict to those $V$ that are affine in the above definition. 

We call the increasing filtration~$W_{\bullet}^{\St}$ on~$H^k(S,\Z/m)$ the \textit{Stein weight filtration}. The following proposition gathers its basic properties.

\begin{prop}
\label{propweight}
Let $S$ be a finite-dimensional Stein space. Fix $k\geq 0$ and $m\geq 1$.
\begin{enumerate}[label=(\roman*)] 
\item 
\label{Wi}
Let $p:T\to S$ be a holomorphic map of finite-dimensional Stein spaces. For~$j\in\Z$, one has $p^*(W^{\St}_jH^k(S,\Z/m))\subset W^{\St}_jH^k(T,\Z/m)$.
\item 
\label{Wii}
One has $W_{-1}^{\St}H^k(S,\Z/m)=0$ and $W_{2k}^{\St}H^{k}(S,\Z/m)=H^k(S,\Z/m)$.
\item
\label{Wiii}
Let $V$ be a quasi-projective algebraic variety over $\C$. Let $f:S\to V^{\an}$ be holomorphic. If $\alpha\in H^k(V^{\an},\Z/m)$ satisfies $f^*\alpha\in W_j^{\St}H^k(S,\Z/m)$, then there exists a morphism $g:V'\to V$ of quasi-projective algebraic varieties over~$\C$ such that $f$ lifts to $f':S\to (V')^{\an}$ and $(g^{\an})^*\alpha\in W_jH^k((V')^{\an},\Z/m)$.
\end{enumerate}
\end{prop}

\begin{proof}
Assertion \ref{Wi} is immediate from the definition. Assertion \ref{Wii} follows from the corresponding properties of the weight filtration (see \cite[Corollary 6.3]{CG}) and from the surjectivity of \eqref{colim}. Assertion \ref{Wiii} follows from the injectivity of~\eqref{colim}.
\end{proof}

\begin{rems} 
(i)
In \cite[\S 2]{TotaroICM}, Totaro extended the construction of the weight filtration on the compactly supported cohomology of algebraic varieties given in~\cite{GS} to the case of complex spaces endowed with an equivalence class of compactifications. The weight filtration on cohomology without support is also defined in this more general setting (see~\mbox{\cite[\S 6]{CG}}). The Stein weight filtration defined above is not a special case of this construction. Indeed, Stein spaces are not compactifiable in general. In addition, when a Stein space is compactifiable, the weight filtration may depend on the choice of compactification (as in Serre's example described in Remark~\ref{remcolim}\,(i)), whereas the Stein weight filtration is entirely canonical.

(ii)
As an example, let us compute $W_{\bullet}^{\St}H^1(\C^\times,\Z/m)$. For any elliptic curve $E$ over~$\C$, there exists a holomorphic map $f:\C^\times\to E^{\an}$ that is a topological covering of group $\Z$. This shows that $W_{1}^{\St}H^1(\C^\times,\Z/m)=H^1(\C^\times,\Z/m)$. In addition, if~$V$ is an algebraic variety over $\C$ and $f:\C^\times\to V^{\an}$ is holomorphic, we let $V'$ be a desingularization of the Zariski closure of the image of $f$. As $f$ lifts to a holomorphic map $f':\C^\times\to (V')^{\an}$ and  $V'$ is smooth, this shows that $W_{0}^{\St}H^1(\C^\times,\Z/m)=0$.

(iii) 
We do not know whether $W_{k}^{\St}H^k(S,\Z/m)=H^k(S,\Z/m)$ holds for all finite-dimensional Stein spaces. If one could restrict, in~\eqref{colim}, to those $f:S\to V^{\an}$ with~$V$ projective, this equality could be deduced from \cite[Corollary~6.3\,2)]{CG}. 

(iv)
Let $S$ be a Stein manifold. We do not know whether $W_{k-1}^{\St}H^k(S,\Z/m)$ always vanishes. If one could restrict, in \eqref{colim}, to those $f:S\to V^{\an}$ with $V$ smooth, this could be deduced from~\cite[Corollary 6.3\,1)]{CG}. In turn, this fact would follow from the (unproven) assertion that $\cO(S)$ is a filtered colimit of smooth $\C$-algebras (a Stein analogue of the N\'eron--Popescu desingularization theorem \cite{Swan}).
\end{rems}
 
\bibliographystyle{myamsalpha}
\bibliography{Steinetale}

\end{document}